\documentclass[12pt]{amsart}
\usepackage{amsthm}
\usepackage[utf8]{inputenc}
\usepackage{appendix}
\usepackage{hyperref}
\usepackage{graphicx}
\usepackage{cite}
\usepackage{amssymb,amsmath,amsthm}
\usepackage{hyperref}
\usepackage[dvipsnames]{xcolor}
\usepackage{url}
\newtheorem{theorem}{Theorem}[section]
\theoremstyle{plain}
\newtheorem{prop}[theorem]{Proposition}
\newtheorem{corollary}[theorem]{Corollary}
\newtheorem{lemma}[theorem]{Lemma}

\theoremstyle{definition}

\newtheorem{ex}[theorem]{Example}

\newtheorem{remark}{Remark}[section]
\theoremstyle{remark}

\def\rn{M}

\def\n{\mathbb{N}}

\def\o{\Omega}

\def\hat{\widehat}
\def\tilde{\widetilde}

\def\cC{\mathcal{C}}

\def\cH{\mathcal{H}}

\def\cN{\mathcal{N}}
\def\cP{\mathcal{P}}
\def\cJ{\mathcal{J}}
\def\cT{\mathcal{T}}
\def\cU{\mathcal{U}}

\def\bf{\mathbf}
\def\N{\mathbb{N}}
\def\R{\mathbb{R}}

\def\cP{\mathcal{P}}

\def\cC{\mathcal{C}}

\def\r{\mathbb{R}}
\newcommand{\hess}{\mathrm{Hess}}

\title[Equivariant solutions to optimal partition problem]{Equivariant solutions to the optimal partition problem for the prescribed $Q$-curvature equation}

\author[J.~C. ~FERNÁNDEZ]{Juan Carlos Fernández$^{\ast}$} 
\address[J.~C.~FERNÁNDEZ]{Facultad de Ciencias, Universidad Na\-cio\-nal Au\-tó\-no\-ma de México (UNAM), Ciudad de Mé\-xi\-co, Mé\-xi\-co.}
\email{\href{mailto:jcfmor@ciencias.unam.mx}{jcfmor@ciencias.unam.mx}}
\thanks{$^{\ast}$Partially supported by CONACYT A1-S-10457 grant and by UNAM-DGAPA-PAPIIT IN101322 grant.}

\author[O.~PALMAS]{Oscar Palmas$^{\ast\ast}$} 
\address[O.~PALMAS]{Facultad de Ciencias, Universidad Na\-cio\-nal Au\-tó\-no\-ma de México (UNAM), Ciudad de Mé\-xi\-co, Mé\-xi\-co.}
\email{\href{mailto:oscar.palmas@ciencias.unam.mx}{oscar.palmas@ciencias.unam.mx}}
\thanks{$^{\ast\ast}$ Supported by UNAM-DGAPA-PAPIIT IN101322 grant.}

\author[J.~TORRES OROZCO]{Jonatán Torres Orozco$^{\ast\ast\ast}$} 
\address[J.~TORRES OROZCO]{Facultad de Ciencias, Universidad Na\-cio\-nal Au\-tó\-no\-ma de México (UNAM), Ciudad de Mé\-xi\-co, Mé\-xi\-co.}
\email{\href{mailto:jonatan.tto@gmail.com}{jonatan.tto@gmail.com}}
\thanks{$^{\ast\ast\ast}$ Supported by a CONACyT postdoctoral fellowship (Mexico) and partially supported by UNAM-DGAPA-PAPIIT IN101322 grant.}
\date{}

\begin{document}

\maketitle

%\author{Juan Carlos Fernández, Oscar Palmas\footnote{\thanks{O. Palmas was supported by UNAM-DGAPA-PAPIIT IN101322 grant.}} and Jonatán Torres-Orozco\footnote{\thanks{ J. Torres Orozco was supported by a CONACyT postdoctoral fellowship (Mexico) and partially supported by UNAM-DGAPA-PAPIIT IN101322 grant.}}}

\begin{abstract}
We study the optimal partition problem for the prescribed constant $Q$-curvature equation induced by the higher order conformal operators under the effect of cohomogeneity one actions on Einstein manifolds with positive scalar curvature. This allows us to give a precise description of the solution domains and their boundaries in terms of the orbits of the action. We also prove the existence of least energy symmetric solutions to a weakly coupled elliptic system of prescribed $Q$-curvature equations under weaker assumptions and conclude a multiplicity result of sign-changing solutions to the prescribed constant $Q$-curvature problem induced by the Paneitz-Branson operator. Moreover, we study the coercivity of $GJMS$-operators on Ricci solitons, compute the $Q$-curvature of these manifolds and give a multiplicity result for the sign-changing solutions to the Yamabe problem with prescribed number of nodal domains on the Koiso-Cao Ricci soliton.  
\end{abstract}

\tableofcontents

% NOTACIÓN:
% $(M,g)$ Riemannian manifold.

% Puntos en $M$ serán $x\in M$, en vez de $p\in M$

% $p$ denotará un número entre 2 y $2_m^\ast$.

% $N$ dimension of $M$.

% $P_g$ the GJMS-operator with respect to $g$.

% \textcolor{red}{$\lambda$ and $\lambda_g$ denote  the Lebesgue measure.
% \newline
% \indent The constant in Einstein metrics and Ricci solitons is $\mu$
% \newline
% \indent The constant appearing in the systems are now $\eta_{ij}$ instead of $\eta_{ij}$ and $\nu_i$ instead of $\nu_i$} 

% $m$ order of $P_g$ 

% $Q_g$ the $Q$-curvature of $g$, for $P_g$.

% $R_g$ the scalar curvature of $g$.

% $\ell$ the amount of prescribed nodal domains.

% $\Gamma$ a closed subgroup of ${\rm Isom}(M, g)$.

% $r$ the distance function. (No $g$ subscript).

% $dV_g$ volume form with respect to g

% $dV_{g_t}$ volume form with respect to $g_t$. It is the volume form at the principal orbits of the action.

% $\mathcal{C}^\infty(X)$ space of smooth function on $X$. 

% $h(t)$ is the mean curvature of the level hypersurfaces of the distance function.

\section{Introduction}

A milestone in Geometric Analysis and Conformal Geometry is the Yamabe problem, which consists in finding a conformal metric on a given Riemannian manifold in order to obtain constant scalar curvature (see, for instance, the book \cite{AubinBook} or the classic survey \cite{LePa}, and the references therein). There are many generalizations of this problem. One direction is to consider higher-order operators generalizing the conformal Laplacian. For instance, a fourth-order operator on Riemannian manifolds of dimension four, which is the analogue to the bilaplacian on the Euclidean space, was discovered by S. Paneitz \cite{Paneitz1983} and later generalized to higher dimension by T.P Branson in \cite {Branson1995}, and by Branson and B. \O{}rsted in \cite{BransonOrsted1991}. A systematic construction of conformally invariant operators of higher order was given by C. R. Graham, R. Jenne, J. Mason and G. A. J. Sparling in \cite{GJMS} ($GJMS$-operators for short) and, more recently extended to fractional order conformally invariant operators by S.-Y. A. Chang and M. d. M. González in \cite{ChangGonzalez2011}. This paper deals with optimal partition problems related to equations given by GJMS-operators.

\smallskip
In order to briefly describe these operators, let $M$ be a closed (compact without boundary) smooth manifold of dimension $N$ and let $m\in\N$ such that $2m< N$.  For any Riemannian metric $g$ for $M$, there exists an operator $P_g:\mathcal{C}^\infty(M)\rightarrow \mathcal{C}^\infty(M)$, satisfying the following properties.
\begin{enumerate}
	\item[($P_1$)] $P_{g}$ is a differential operator and $P_{g}=(-\Delta_g)^m + \emph{\text{lower order terms}}$, where $\Delta=\text{div}_g(\nabla_g)$ is the Laplace Beltrami operator on $(M,g)$.
	\item[($P_2$)]\label{Propoerty:Pg naturality} $P_g$ is natural, i.e., for every diffeomorphism $\varphi:M\rightarrow M$,  $\varphi^\ast P_{g}=P_{\varphi^\ast g}\circ\varphi^\ast$, where ``$\!\!\!\phantom{K}^\ast\;$'' denotes the pullback of a tensor. 
	\item[($P_3$)] $P_{g}$ is self-adjoint with respect to the $L^2$-scalar product.
	\item[($P_4$)]\label{Property:Pg conformal invaiance} $P_g$ is conformally invariant, that is, given any function $\omega\in\mathcal{C}^\infty(M)$, if we define the conformal metric $\widetilde{g}=e^{2\omega} g$, then the following identity holds true:
	\begin{equation}\label{Eq:GJMS-conformal invariance}
		P_{\widetilde{g}}(f)=e^{-\frac{N+2m}{2}\omega}P_{g}\left(e^{\frac{N-2m}{2}\omega}f\right), \text{ for all }f\in\mathcal{C}^\infty(M).   
	\end{equation}
	%	\item $P_{g}$ has a polynomial expression in $\nabla$ and $R_g$, in which all coefficients are rational in the dimension $N$.
	%\item If $d$ and $\delta$ are the de Rham operators and $\Delta$ is  the form Laplacian $\delta d + d\delta$, then $P_{g}$ has the form
	%\[
	%\delta S_{g,m,k} d + \frac{m-2k}{2}Q_{g,m,k}
	%\]
	%	where $Q_{g,m,k}$ is a local scalar invariant and $S_{g,m,k}$ is an operator on 1-forms which can be expressed in the form $\Delta^{k/2 - 1}+ lot$.
\end{enumerate}
\medskip

There is a natural conformal invariant associated with $P_g$ given by
\[
Q_g:=\frac{2}{N-2m}P_g(1),
\]
called the {\em Branson $Q$-curvature}, after Branson-{\O}rsted \cite{BransonOrsted1991} and Branson \cite{BransonBook,Branson1995}, or simply the {\em $Q$-curvature}. When $m=1$, $P_g$ is the conformal Laplacian and $Q_g$ is the scalar curvature $R_g$, while for $m=2$, $P_g$ is the Paneitz-Branson operator and $Q_g$ is the usual $Q$-curvature \cite{DjadliHebeyLedoux2000}. 
\newline

When considering conformal metrics $\tilde{g}=u^{4/(N-2m)}g$ with $u>0$ in $\mathcal{C}^\infty(M)$, equation \eqref{Eq:GJMS-conformal invariance} is written as
\begin{equation}\label{Eq:GJMS-Identity}
	P_{\widetilde{g}}\phi = u^{1-2_m^\ast}P_g(u\phi),
\end{equation}
where $2_m^\ast:=\frac{2N}{N-2m}$ is the critical Sobolev exponent of the embedding $H_g^m(M)\hookrightarrow L^p(M,g)$. Here $H_g^m(M)$ denotes the Sobolev space of order $m$, which is the closure of $\mathcal{C}^\infty(M)$ in $L^2_g(M)$ under the norm \begin{equation}\label{Eq:StandardSobolevNorm}
    \Vert u\Vert_{H^m}:=\left(\sum_{i=1}^m \int_M\vert \nabla_g^{i}u\vert^2 dV_g\right)^{1/2}.
\end{equation}
Taking $\phi\equiv 1$ in (\ref{Eq:GJMS-Identity}), one obtains the \emph{prescribed $Q$-curvature equation}
\begin{equation}\label{Eq:Q-curvature equation}
	P_g u = \frac{N-2m}{2}Q_{\tilde{g}}u^{2_m^\ast-1},\quad \text{ on } M.
\end{equation}
For $m=1$ and $Q_{\tilde{g}}$ constant, we recover the Yamabe Problem. Some results about the existence and multiplicity of solutions to the prescribed $Q$-curvature problem can be found in \cite{AzaizBoughazi2020,BaScWe,BenaliliBoughazi2016,DjadliHebeyLedoux2000,Ro,Tahri2020}. 
\newline
\smallskip 

Given $\ell\in\mathbb{N}$, we will deal with the \emph{$\ell$-partition problem} associated with \eqref{Eq:Q-curvature equation} in the presence of symmetries. In order to describe the problem, let $\Gamma$ be a closed subgroup of Isom$(M,g)$ under suitable conditions (see hypotheses (\hyperref[Gamma:Cohomogeneity]{$\Gamma1$}) and (\hyperref[Gamma:DimensionOrbits]{$\Gamma2$}) below), and let $\Omega\subset M$ be an open and $\Gamma$-invariant subset, namely, if $x\in\Omega$, then $\gamma x\in\Omega$ for every isometry $\gamma\in\Gamma$. In what follows, $H_{0,g}^m(\Omega)$ denotes the closure of $\mathcal{C}_c^\infty(\Omega)$ under the Sobolev norm given by (\ref{Eq:StandardSobolevNorm}).
%\begin{equation}
%\Vert u\Vert_{H^m} = \left( \sum_{i=1}^m \int_\Omega \vert \nabla^i_g u\vert^2 \right)^{1/2}.
%\end{equation}
\newline

We consider the symmetric Dirichlet boundary problem
\begin{equation} \label{eq:dirichlet}
	\begin{cases}
		P_g u = |u|^{2_m^*-2}u, & \text{ in } \Omega,\\
		\nabla^{i}u =0, i=0,1,\ldots,2m-1, & \text{ on } \partial\Omega, \\
		u \text{ is }\Gamma\text{-invariant},
	\end{cases}
\end{equation}
where $u:\Omega\rightarrow\mathbb{R}$ is said to be \emph{$\Gamma$-invariant} if $u(x)=u(\gamma x)$ for every $x\in M$ and any $\gamma\in\Gamma$. 

Denote by $c_\Omega^\Gamma$ the least energy among the nontrivial solutions, that is,
\[
c^\Gamma_\Omega:=\inf\left\{\frac{m}{N}\int_{\Omega}\vert u\vert^{2^\ast_m} dV_g:u\neq 0\text{ and } u\text{ solves }\eqref{eq:dirichlet}\right\}.
\]
In the absence of symmetries, the lack of compactness due to the critical Sobolev exponent nonlinearity in equation \eqref{eq:dirichlet}, implies that this number is not achieved in general  \cite[Chapter III]{StruweBook}.  However, when the domain is smooth, $\Gamma$-invariant and the $\Gamma$-orbits have positive dimension, this number is attained (see Proposition \ref{Proposition:ExistenceLeastEnergySolution} below). The $\ell$-partition problem consists in finding mutually disjoint and non empty $\Gamma$-invariant open subsets $\Omega_1,\ldots,\Omega_\ell$ such that
\begin{equation}\label{Problem:PartitionProblem}
\sum_{i=1}^\ell c_{\Omega_i}^\Gamma \leq \inf_{(\Phi_1,\ldots,\Phi_\ell)\in \mathcal{P}_\ell^\Gamma} \sum_{i=1}^\ell c_{\Phi_i}^\Gamma \end{equation}
where
\begin{align*}	\mathcal{P}_\ell^\Gamma:=\{\{\Omega_1,\ldots,\Omega_\ell\}:&\,\Omega_i\neq\emptyset \text{ is a }\Gamma\text{-invariant open subset of }M\text{ and }\Omega_i\cap \Omega_j=\emptyset\text{ if }i\neq j \}.
\end{align*}
\medskip

The aim of this paper is to show that this problem has a solution on Einstein manifolds  with positive scalar curvature, for every $m\geq 2$, and with less restrictive hypotheses on the metric, when $m=1$. 
\newline

In order to state our main result, we need to impose some conditions over $(M,g)$ and its isometry group (conditions \emph{\hyperref[Gamma:Cohomogeneity]{$(\Gamma1)$}} to \emph{\hyperref[Gamma:MetricDecomposition]{$(\Gamma3)$}} below). For the reader convenience, we recall some basic facts about isometric actions (see  Chapter 3 and Section 6.3 in \cite{AlexBettiol} for a detailed exposition). For any $x\in M$, the $\Gamma$ orbit of $x$ is the set $\Gamma x:=\{\gamma x\;:\;\gamma\in\Gamma\}$, and the isotropy subgroup of $x$ is defined as $\Gamma_x:=\{\gamma\in\Gamma \;:\; \gamma x = x\}$. When $\Gamma$ is a closed subgroup of $\text{Isom}(M,g)$, the $\Gamma$-orbits are closed submanifolds of $M$ which are \emph{$\Gamma$-diffeomorphic} to the homogeneous space $\Gamma/\Gamma_x$, meaning the existence of a $\Gamma$-invariant diffeomorphism between these two manifolds. An orbit $\Gamma x$ is called \emph{principal}, if there exists a neighborhood $V$ of $x$ in $M$ such that for each $y\in V$, $\Gamma_x\subset \Gamma_{\gamma y}$ for some $\gamma\in\Gamma$. All points lying in a principal orbit have, up to conjugacy, the same isotropy group. Denote this group by $K$, called the \emph{principal isotropy group.} We say that $\Gamma$ induces a \emph{cohomogeneity one action} on $M$ if the principal orbits have codimension one. In presence of a cohomogeneity one action, each principal $\Gamma$-orbit is a closed hypersurface in $M$ diffeomorphic to $\Gamma/K$, and there are exactly two orbits of bigger codimension, called \emph{singular orbits}. Denote them by $M_-$ and $M_+$. $M_\pm$ are closed submanifolds of codimension $\geq 2$, and all points lying in $M_\pm$ have, up to conjugacy, the same isotropy group. 
%Then $M^{o}=M\setminus (M_-\cup M_+)$.
Denoting these groups by $K_\pm$, we have that $M_\pm$ is  $\Gamma$-diffeomorphic to $\Gamma/K_\pm$.
\newline

We can now state the conditions on the group $\Gamma$. In what follows, $d$ will denote the geodesic distance between $M_+$ and $M_-$, that is, $d:=\text{dist}_g(M_-,M_+)$.

 \begin{enumerate}
    \item[$(\Gamma 1)$] \label{Gamma:Cohomogeneity} $\Gamma$ is a closed subgroup of $\text{Isom}(M,g)$ inducing a cohomogeneity one action on $M$. 
    \item[$(\Gamma 2)$] \label{Gamma:DimensionOrbits} $1\leq \dim \Gamma x\leq N-1$ for every $x\in M$.
    
    \item[$(\Gamma3)$] \label{Gamma:MetricDecomposition} 
    The metric $g$ on $M$ is a $\Gamma$-invariant metric of the form:
    \begin{equation*}
    %\label{MetricDecomposition}
    g=dt^2+\sum_{j=1}^k f_j^2(t)\ g_j,    \end{equation*}
    where $g_t:=\sum_{j=1}^k f_j^2(t)\ g_j$ is one parameter family of metrics on the principal orbit, for some positive smooth functions defined on the interval $I=[0, d]$, with suitable asymptotic conditions at $0$ and $d$. More precisely, the (smooth) metrics $g_j$ are defined on the principal orbit at $t$, for $t\in(0, d)$, and are defined around the singular orbits in such a way that $f_j$, $j=1, \dots, k$, satisfy appropriate smoothness conditions at the endpoints $t_-=0$ and $t_+=d$ of $I$, that ensure that $g$ can be extended to the singular orbits:    
    \begin{equation}
    \label{SmoothnessCond}
f_1(t_{\pm})=0, f_1'(t_{\pm})=1;\ f_j(t_{\pm})>0, f_j'(t_{\pm})=0\ \mbox{for $1<j\leq k$}.    
    \end{equation}    
     See \cite[Section 1]{EscWang00}.
\end{enumerate}
\medskip

In Section \ref{Section:Examples}, we give concrete examples where these hypotheses hold true.
\newline

We are now ready to state our main result, which describes the solution to the $\ell$-optimal partition problem in terms of the principal orbits of cohomogeneity one actions. The symbol ``$\approx$'' will stand for ``$\Gamma$-diffeomorphic to''. 

\begin{theorem}\label{Theorem:OptimalPartitionSymmetry}
Let $m\in\mathbb{N}$ and let $(M,g)$ be a closed Riemannian manifold of dimension $N>2m$ with scalar curvature $R_g$. If 
\begin{itemize}
    \item $R_g>0$ for $m=1$; or
    \item $(M,g)$ is Einstein with $R_g>0$,
\end{itemize}
then for any $\ell\in \mathbb{N}$ and any subgroup $\Gamma$ of Isom$(M,g)$ satisfying  \emph{(\hyperref[Gamma:Cohomogeneity]{$\Gamma1$})}, \emph{(\hyperref[Gamma:DimensionOrbits]{$\Gamma2$})} and \emph{(\hyperref[Gamma:MetricDecomposition]{$\Gamma3$})}, the $\Gamma$-invariant $\ell$-partition problem \eqref{Problem:PartitionProblem} has a solution $(\Omega_1,\ldots,\Omega_\ell)\in\mathcal{P}_\Omega^\Gamma$ such that:
\begin{enumerate}
    \item $\Omega_i$ is connected for every $i=1,\ldots, \ell$, $\overline{\Omega}_i\cap\overline{\Omega}_{i+1}\neq\emptyset$, $\Omega_i\cap\Omega_j=\emptyset$ if $\vert i-j \vert\geq 2$ and $\overline{\Omega_1\cup\ldots\cup\Omega_\ell }= M$.
    \item The sets $\Omega_1$ and $\Omega_\ell$ are $\Gamma$-diffeomorphic to disk bundles at $M_{-}$ and $M_{+}$, respectively. More precisely, 
    %    if $K_{\pm}$ denote the singular isotropies, and $D_{\pm}$ are normal slices to the two singular orbits:
    \[
    \Omega_1\approx \Gamma\times_{K-}D_{-},\quad \Omega_\ell\approx \Gamma\times_{K+}D_{+}, \quad \partial \Omega_1\approx\partial\Omega_\ell\approx \Gamma/K,
    \]
    where $\Gamma\times_{K_{\pm}}D_{\pm}$ are normal bundles over the singular orbits $M_{\pm}$. See Section \ref{Section:Segregation} for details.
    
    \item For each $i\neq 1,\ell$, $\Omega_i\approx \Gamma/K\times(0,1)$, $\overline{\Omega}_i\cap\overline{\Omega}_{i+1}\approx\Gamma/K$ and $\partial\Omega_i\approx \Gamma/K \sqcup \Gamma/K$, where $\sqcup$ denotes the disjoint union of sets.\end{enumerate}
\end{theorem}

In Section \ref{Section:Examples} we give several examples of Einstein manifolds with positive scalar curvature satisfying conditions (\hyperref[Gamma:Cohomogeneity]{$\Gamma1$}), (\hyperref[Gamma:DimensionOrbits]{$\Gamma2$}) and (\hyperref[Gamma:MetricDecomposition]{$\Gamma3$}). This result extends Theorem 1.1 in \cite{CSS21} and Theorem 1.2 in \cite{ClappFernandezSaldana2021} to more general manifolds and actions than the sphere with the $O(n)\times O(k)$-action, $n+k=N+1$.  In the case of $m=1$, M. Clapp and A. Pistoia \cite{ClappPistoia21} solved the $\Gamma$-invariant $\ell$-partition problem for actions with higher cohomogeneity and in \cite{ClappPistoiaTavares21}, the authors solved the $\ell$-partition problem without any symmetry assumption. However, neither the structure nor the domains solving the problem are explicit in these works. 

Notice that in case $m=1$, the scalar curvature is $\Gamma$-invariant, since we are regarding isometric actions. An interesting application of Theorem \ref{Theorem:OptimalPartitionSymmetry} in this case, is the following result for the Yamabe problem.

\begin{corollary}\label{Corollary:YamabeProblem}
Let $(M,g)$ be a closed Riemannian manifold of dimension $N\geq3$ and let $\Gamma$ be a closed subgroup of $\text{Isom}(M,g)$ satisfying \emph{\hyperref[Gamma:Cohomogeneity]{$(\Gamma1)$}} to \emph{\hyperref[Gamma:MetricDecomposition]{$(\Gamma3)$}}. If $R_g>0$, then for any $\ell\in\mathbb{N}$, the Yamabe  problem 
\[
-\Delta_g u + \frac{N-2}{4(N-1)}R_g u = \vert u\vert^{2_1^\ast-2}u,\quad \text{on } M,
\]
admits a $\Gamma$-invariant sign changing solution with exactly $\ell$-nodal domains. Moreover, it has least energy among all such solutions and the nodal set is a disjoint union of $\ell-1$ principal orbits.
\end{corollary}
This was proven in the case $\ell=2$ in \cite{ClappPistoia21} for general manifolds in the presence of symmetries, and in \cite{ClappPistoiaTavares21} in the non-symmetric case; for any $\ell\in\mathbb{N}$, the only manifold where this was known is  the round sphere in the presence of symmetries given by isoparametric functions \cite{FdzPetean20,CSS21}.

This corollary clearly holds on Einstein manifolds with positive scalar curvature, for which the scalar curvature is constant. Then, a natural setting for extending its applications is Ricci soliton metrics. Recall that a {\em Ricci soliton} on a closed smooth manifold $M$ is a Riemannian metric $g$ satisfying the equation
\begin{equation}
\label{eqn: RicciSoliton}
Ric(g)+{\rm Hess}_g(f)=\mu g.
\end{equation}
for some constant $\mu=-1, 0,$ or $1$, and some smooth function $f$, called the \emph{Ricci potential}. Here, as usual, $Ric(g)$ denotes the Ricci curvature of $g$, and ${\rm Hess}_g(f)$ denotes the Hessian of $f$ with respect to $g$. A Ricci soliton is called \emph{steady, shrinking} or \emph{expanding} according to $\mu = 0$, $\mu>0$ or $\mu<0$, respectively. 

In these terms, a nontrivial explicit example that fulfills the hypotheses of Corollary \ref{Corollary:YamabeProblem} and that does not reduce to an Einstein metric, is the {\em Koiso-Cao Ricci soliton} \cite{Koiso, Cao1}. It is known that this metric is a cohomogeneity one Kähler metric on  $\mathbb{CP}^2\#\overline{\mathbb{CP}^2}$, with respect to the action of U$(2)$, the unitary group of dimension $2$. Here, as usual, $\# $ denotes the connected sum of smooth manifolds  and $\overline{M}$ denotes a smooth manifold $M$ taken with reverse orientation.  It is also known that it is  non-Einstein with positive Ricci curvature. The singular orbits of the U$(2)$-action are both diffeomorphic to $\mathbb{S}^2$, and the principal orbits are diffeomorphic to $\mathbb{S}^3$.  The associated fibrations, as a cohomogeneity one manifold, are both given by the Hopf fibration:
\[
{\rm U}(1)\to {\rm SU}(2)\to {\rm SU}(2)/{\rm U}(1).
 \]
Here, for an orientable manifold $M$, we will denote by $\overline M$ to refer to the opposite orientation. 

See \cite[Section 2]{TOR17} for details on the construction of the Koiso-Cao soliton from the Hopf fibration. 
\newline

In order to establish the existence of a solution to the optimal partition problem \eqref{Problem:PartitionProblem}, we follow the approach introduced in \cite{ContiTerraciniVerzini2002}, consisting in the study of a weakly coupled competitive system together with a segregation phenomenon. To this end, we will study the existence of $\Gamma$-invariant solutions to the following weakly coupled competitive $Q$-curvature system

\begin{equation}\label{Eq:Q Systems}
	P_g u_i = \nu_i\vert u_i\vert^{2_m^\ast-2}u_i + \sum_{i\neq j}\eta_{ij}\beta_{ij}\vert u_j\vert^{\alpha_{ij}} \vert u_i\vert^{\beta_{ij}-2}u_i,\quad\text{ on }M,\ \   i,j=1,\ldots,\ell,
\end{equation}
where $\nu_i>0$, $\eta_{ij}=\eta_{ji}<0$ and $\alpha_{ij},\beta_{ij}>1$ are such that $\alpha_{ij}=\beta_{ji}$ and $\alpha_{ij}+\beta_{ij}=2_m^\ast$. We will say that a solution $\overline{u}=(u_1,\ldots,u_\ell)$ to the system \eqref{Eq:Q Systems} is \emph{fully nontrivial}, if, for each $i=1,\ldots,\ell$, $u_i$ is non trivial.

To assure the existence of a fully nontrivial least energy $\Gamma$-invariant solution to the system \eqref{Eq:Q Systems} we will only assume that $\Gamma$ satisfies (\hyperref[Gamma:DimensionOrbits]{$\Gamma2$}), allowing actions with bigger codimension of its principal orbits, and also we will assume that the operator $P_g$ is \emph{coercive}, meaning the existence of a constant $C>0$ such that
\begin{equation*}%\label{Eq: Coercivity}
	\int_M uP_gu\; dV_g \geq C\Vert u\Vert_{H^m}^2,\qquad\text{ for every }u\in\mathcal{C}^\infty(M).
\end{equation*}
We will say that a sequence of fully nontrivial elements $\overline{u}_k$ in the Sobolev space
\[
H_g^m(M)^\ell:=\underbrace{H_g^m(M)\times\cdots\times H_g^m(M)}_{\ell\text{ times}}
\]
is \emph{unbounded} if $\Vert u_{k,i}\Vert_{H^m}\rightarrow\infty,$ as $k\rightarrow\infty$,  for every $i=1,\ldots,\ell.$ In this direction, we state our next multiplicity result.

\begin{theorem}\label{Th:MainQSystems}
	If the operator $P_g$ is coercive and \emph{(\hyperref[Gamma:DimensionOrbits]{$\Gamma2$})} holds true, then the system \eqref{Eq:Q Systems} admits an unbounded sequence of $\Gamma$-invariant fully nontrivial solutions. One of them has least energy among all fully nontrivial $\Gamma$-invariant solutions. 
\end{theorem}

When $\ell=1$ and $m=1$, by a well-known argument given in \cite[Proof of Theorem A]{BenciCerami1991}, the least energy solutions are positive and, hence, they give rise to a $\Gamma$-invariant solution to the Yamabe problem. Seeking for this kind of metrics is a classical problem posed by E. Hebey and M. Vaugon in \cite{HebeyVaugon1993}. However, there is a gap in Hebey and Vaugon's proof, for they used Schoen's Weyl vanishing conjecture, which turns out to be false in higher dimensions by Brendle's counterexample in \cite{Brendle2008}. F. Madani solved the equivariant Yamabe problem in \cite{Madani2012} assuming the positive mass theorem to construct good test functions. Here, the positive dimension of the orbits given by hypothesis (\emph{\hyperref[Gamma:DimensionOrbits]{$\Gamma2$}}), avoids these problems in higher dimensions. For $m\geq 2$, the Maximum Principle for the operator $P_g$ is not true in general, and the least energy solutions may change sign. In fact, it is not clear whether the components of the solutions to the system \eqref{Eq:Q Systems} are sign-changing or not. In case $m=2$ and $\ell=1$, by a recent result by J. Vétois \cite[Theorem 2.2]{Vetois2022}, we can assure that an infinite number of the solutions to
\begin{equation}\label{Equation:Paneitz-Branson}
P_g u = \vert u\vert^{2_2^\ast - 4}u\qquad \text{in } M,
\end{equation}
must change sign when considering $(M,g)$ to be Einstein with positive scalar curvature, as we next state.

\begin{corollary}\label{Corollary:NodalSolutionsPaneitz}
Let $m=2$, $(M,g)$ be an Einstein manifold of dimension $N>4$, with positive scalar curvature  and not conformally diffeomorphic to the standard sphere. If $\Gamma$ is a closed subgroup of $\text{Isom}(M,g)$ satisfying \emph{(\hyperref[Gamma:DimensionOrbits]{$\Gamma2$})}, then the problem with Paneitz-Branson operator \eqref{Equation:Paneitz-Branson}
admits an unbounded sequence of $\Gamma$-invariant sign-changing solutions.
\end{corollary}
% \begin{proof}[Proof of Corollary \ref{Corollary:NodalSolutionsPaneitz}.]
%  Theorem 2.2 in \cite{Vetois2022} states that the positive solutions to the problem \eqref{Equation:Paneitz-Branson} must be constant and by the concrete expression of the Paneitz-Branson operator and the $Q$-curvature on Einstein manifolds (see, for instance, \cite{DjadliHebeyLedoux2000} and \cite{Gover06} respectively) it is unique. As $(M,g)$ is Einstein with positive scalar curvature, the operator $P_g$ is coercive \cite[Proposition 4]{Ro} and Theorem \ref{Th:MainQSystems} for $\ell=1$, yields the existence of an unbounded sequence of $\Gamma$-invariant solutions, and the corollary follows.  
% \end{proof}
\medskip

There are many examples of manifolds admitting an Einstein metric $g$ with positive scalar curvature, in order to ensure the coercivity of $P_g$. In Section \ref{Section:Examples}, we describe explicit examples and some construction that provide large classes of examples. However, it is difficult to find nontrivial examples of non-Einstein manifolds for which the operator $P_g$ is coercive. Towards this direction, the next result for the Paneitz-Branson operator gives a sufficient condition for this to happen.

\begin{prop}\label{Proposition:Coercivity}
For $m=2$ and $N\geq 6$, if $Q_g>0$ and $R_g>0$, then $P_g$ is coercive.
\end{prop}
\begin{proof}
It is a direct consequence of a calculation obtained in \cite{GurskyMal03}. It appears as equation (2.18): 
\[
\begin{split}
&\int_M \phi P_g \phi\\
&= \int_M\left[\frac{N-6}{N-2}(\Delta \phi)^2+\frac{4|{\rm Hess}(\phi)|^2 }{N-2}+\frac{(N-2)^2+4}{2(N-1)(N-2)}R_g|\nabla \phi|^2+\frac{N-4}{2} Q_g\phi^2 \right] d{\rm vol}(g).
\end{split}
\]
\end{proof}

Again, it is complicated to compute the $Q$-curvature of an arbitrary manifold and the literature lacks examples, different from Einstein manifolds.  We explore the possibility of obtaining positive $Q$-curvature and coercivity for more general manifolds, such as Ricci solitons.  This is difficult to check because almost nothing is known about the Ricci curvature tensor in Ricci solitons. Our next result gives an explicit formula of the $Q$-curvature of these metrics.

\begin{theorem}\label{Theorem:Q-curvatureRicciSolitons}
The $Q$-curvature of a shrinking Ricci soliton $(M, g)$, with Ricci potential $f$, is given by:
%\[ 
%Q_g= -\frac{2{\bf a}+{\bf c}}{2}\Delta R_g+{\bf b}R_g^2+{\bf c} R_g+\frac{{\bf c}}{4}Ric_g(\nabla u, \nabla u),
%\]
\[ 
Q_g=(2{\bf a}-{\bf c})|Ric_g|_g^2+{\bf b}R_g^2-2{\bf a}R-2{\bf a}\ Ric_g(\nabla f, \nabla f).
\]

where
\[
{\bf a}=\frac{1}{2(N-1)}, \quad {\bf b}=\frac{N^3-4N^2+16N-16}{8(N-1)^2(N-2)^2},\quad  {\bf c}=\frac{2}{(N-2)^2}.
\]
If $(M, g)$ has radially positive Ricci curvature, i.e., $Ric_g(\nabla f, \nabla f)>0$, and
\begin{itemize}
\item for $N=4$, $R^2>|Ric_g|_g^2+2R_g+2 Ric_g(\nabla f, \nabla f)$,

\item or for $N>4$, $(2{\bf a}-{\bf c})|Ric_g|_g^2>2{\bf a}\ Ric_g(\nabla f, \nabla f)>0$ and $R_g> 2{\bf a}/{\bf b}$,

\end{itemize}
then $Q_g>0$. 
\end{theorem}
\medskip

\begin{remark}
The condition of having radially positive Ricci curvature is plausible, due to the fact that if $Ric_g(\nabla f, \nabla f)\leq 0$ everywhere, then the Ricci soliton is trivial \cite[Theorem 1.1]{PetWyl09a}. Moreover, if it is radially flat, i.e., $Ric_g(\nabla f, \nabla f)\equiv 0$, then it is also trivial \cite[Proposition 2]{PetWyl09b}). If $(M, g)$ is a cohomogeneity one Kähler-Ricci soliton under the action of $\Gamma$, then one may average over $\Gamma$ to obtain a Ricci potential that is a $\Gamma$-invariant function on $M$. %Thus by taking the trace with respect to $g$ in the Ricci soliton equation, we have that the scalar curvature must be a $\Gamma$-invariant function.  
Then the following identity holds true:
\[
Ric_g(\nabla f, \nabla f)=Ric_g\left(f'\frac{\partial }{\partial t}, f'\frac{\partial }{\partial t} \right)=(f')^2 Ric_g\left(\frac{\partial }{\partial t}, \frac{\partial }{\partial t}\right).
\]
Since a non-trivial shrinking Kähler-Ricci soliton has positive Ricci curvature, then under a cohomogeneity one action, it also has radially positive scalar curvature. Even more, all the known compact non-Einstein Ricci solitons are Kähler\cite{Cao2010}. \qed
\end{remark}
\bigskip

We have the following immediate consequence of Theorems \ref{Th:MainQSystems} and \ref{Theorem:Q-curvatureRicciSolitons}, together with Proposition \ref{Proposition:Coercivity}.

\begin{corollary}%\label{Coro:Einstein Manifolds QSystems} 
	Let $(M,g)$ be Einstein with $R_g>0$  or a shrinking Ricci soliton with radially positive Ricci curvature of dimension $N>2m$. If \emph{(\hyperref[Gamma:DimensionOrbits]{$\Gamma2$})} holds true, then the system \eqref{Eq:Q Systems} admits an unbounded sequence of fully nontrivial $\Gamma$-invariant solutions, one of them with least energy among all fully nontrivial $\Gamma$-invariant solutions.
\end{corollary}

\smallskip 

Our work is organized as follows. In Section \ref{Sec:Variational Setting}, we describe the variational setting in order to prove the existence of least energy $\Gamma$-invariant solutions to the problem \eqref{eq:dirichlet}. Next, in Section \ref{sec:system}, we give the variational setting to study system \eqref{Eq:Q Systems} and prove Theorem \ref{Th:MainQSystems} and Corollary \ref{Corollary:NodalSolutionsPaneitz}. In Section \ref{Section:OneDimensionalReduction}, we describe how hypotheses \emph{\hyperref[Gamma:Cohomogeneity]{$(\Gamma1)$}} and \emph{\hyperref[Gamma:MetricDecomposition]{$(\Gamma3)$}} allow us to reduce the partition problem to a simpler one dimensional problem. In Section \ref{Section:Segregation}, we study the segregation phenomenon that gives the description of the domains and $\Gamma$-invariant functions that solve the optimal partition problem, showing Theorem \ref{Theorem:OptimalPartitionSymmetry}. As an application of Theorem \ref{Theorem:OptimalPartitionSymmetry} with $m=1$, we prove Corollary \ref{Corollary:YamabeProblem}. We give several examples where hypotheses \emph{\hyperref[Gamma:Cohomogeneity]{$(\Gamma1)$}} to \emph{\hyperref[Gamma:MetricDecomposition]{$(\Gamma3)$}} hold true in Section \ref{Section:Examples}. Finally, in Section \ref{Section:Q-curvatureRicciSolitons}, we compute the $Q$-curvature of shrinking Ricci solitons, proving Theorem \ref{Theorem:Q-curvatureRicciSolitons}.

\section{Symmetries and least energy solutions}\label{Sec:Variational Setting}

In this section, we study the existence of least energy  solutions to problem \eqref{eq:dirichlet}.

From now on, $\Omega$ will denote either $M$ or an open, connected $\Gamma$-invariant subset of $M$ with smooth boundary, $\Vert\cdot\Vert_p$ will denote the usual norm in $L_g^p(\Omega)$, $p\geq1$. For $u\in \mathcal{C}^\infty(M)$, the $k$-th covariant derivative of $u$ will be denoted by $\nabla^ku$ and we define its norm as the function $\vert \nabla^ku\vert_g:M\rightarrow\mathbb{R}$ given by
\[
\vert \nabla^ku\vert_g^2 := \nabla^{\alpha_1}\cdots\nabla^{\alpha_k}\nabla_{\alpha_1}\cdots\nabla_{\alpha_k}u, 
\]
where we used the Einstein notation convention.

The Sobolev space $H_{0,g}^m(\Omega)$ is the closure of $\mathcal{C}_c^\infty(\Omega)$ under the norm
\[
\Vert u\Vert_{H^{m}}:=\left(\sum_{i=0}^m \Vert\nabla^i u\Vert_2^2\right)^{1/2} = \left(\sum_{i=0}^m \int_\Omega \vert\nabla^i u\vert_g^2 dV_g\right)^{1/2},
\]
where, with some abuse of notation, $\Vert \nabla^{i}u\Vert_2:=\Vert \, \vert \nabla^{i}u\vert_g \Vert_2$
Notice that in case $\Omega=M$, then $\mathcal{C}_c^\infty(M)=\mathcal{C}^\infty(M)$ and $H_{0,g}^m(M)=H_g^m(M)$; if $\Omega\neq M$, then $H_{0,g}^m(\Omega)$ is a closed subspace of $H^m_g(\Omega)$. 
\newline

Let $P_g$ be the corresponding GJMS-operator in $(M,g)$ of order $m$. For each $k\in\{0,1,\ldots,m-1\}$, there exists a symmetric $T^0_{2k}$-tensor field on $M$, which we will denote by $A_{(k)}(g)$, such that the operator $P_g$ can be written as
\[
P_g = (-\Delta)^m_g + \sum_{k=0}^{m-1}(-1)^{k}\nabla^{j_k\cdots j_1}\left( A_{(k)}(g)_{i_k\cdots i_1 j_1\cdots j_k}\nabla^{i_1\cdots i_k} \right),
\]
where the indices are raised via the musical isomorphism. In particular, for any $u,v\in \mathcal{C}_c^\infty(\Omega)$, integration by parts yields that
\[
\begin{split}
&\int_\Omega uP_g vdV_g \\
&= \begin{cases}
	\int_\Omega\left[ \Delta_g^{m/2}u\Delta^{m/2}v + \sum_{k=0}^{m-1}A_k(g)(\nabla_g^k u,\nabla_g^k v) \right] dV_g, & m \text{ even}, \\
	\int_\Omega\left[\langle \nabla_g\Delta_g^{(m-1)/2}u, \nabla_g\Delta_g^{(m-1)/2}v\rangle_g + \sum_{k=0}^{m-1}A_k(g)(\nabla_g^k u,\nabla_g^k v) \right] dV_g, & m \text{ odd}.
\end{cases}
\end{split}
\]
See \cite[Proposition 1]{Ro} for the details.

As a consequence, the bilinear form $(u,v)\mapsto\int_{\Omega}uP_gv dV_g$ can be extended to a continuous symmetric bilinear form on $H^m_{0,g}(\Omega)$. {When $P_g$ is coercive, this bilinear form is actually a well-defined interior product on $H_{0,g}^m(\Omega)$ that induces a norm equivalent to $\Vert\cdot\Vert_{H^m}$ (see \cite[Proposition 2]{Ro}). We will denote this interior product and norm by $\langle\cdot,\cdot\rangle_{P_g}$ and $\Vert\cdot\Vert_{P_g}$ respectively, and endow $H_g^m(\Omega)$ with it in what follows. Notice that, by definition,
\[
\langle u,v \rangle_{P_g}:= \int_\Omega u P_g v dV_g, \quad\text{and}\quad \Vert u\Vert^2_{P_g}=\int_\Omega u P_g u dV_g, 
\]
for every $u,v\in\mathcal{C}^\infty(\Omega)$.
\newline

The group $\text{Isom}(M,g)$ acts on $H_g^m(M)$ in the usual way: 
\[
\gamma u:= u\circ\gamma^{-1},\qquad u\in H_g^m(M),\   \gamma\in\text{Isom}(M,g).
\]
Every element $\gamma\in\text{Isom}(M,g)$ induces a linear map
\[
\gamma:H_g^{m}(M)\rightarrow H_g^m(M),\quad u\mapsto u\circ\gamma^{-1}.
\]

We next show that the norm is invariant under the action of isometries.

\begin{lemma}\label{Lemma:Pg symmetry invariance}
	For every $\gamma\in\text{Isom}(M,g)$ and every $u\in\mathcal{C}^\infty(M)$
	\[
	P_g(u\circ\gamma)=(P_{g}u)\circ\gamma.
	\]
	In particular, $\gamma:H_g^{m}(M)\rightarrow H_g^m(M)$ is a linear isometry.
\end{lemma}

\begin{proof}
	Let $u\in\mathcal{C}^\infty(M)$ and $\gamma\in\text{Isom}(M,g)$. Then $\gamma:M\rightarrow M$ is a diffeomorphism and, by the naturality of $P_g$, property \hyperref[Propoerty:Pg naturality]{($P_2$)} in the introduction, we obtain for any isometry $\gamma\in \text{Isom}_g(M,g)$ that
	\[
	P_g(u\circ\gamma) = (P_g\circ\gamma^\ast)(u)=P_{\gamma^\ast g}\circ\gamma^\ast(u)=(\gamma^\ast\circ P_g)(u)=P_g(u)\circ\gamma,
	\]
	where $\gamma^\ast g$ denotes the pullback metric.
	
	Now, to see that $\gamma$ induces a linear isometry, recall that if $\gamma\in\text{Isom}(M,g)$, then 
	\begin{equation}\label{Eq:Integral invariance isometries}
		\int_M u\circ\gamma\; dV_g = \int_M u\;  dV_g,
	\end{equation}
	(Cf. \cite[Théorème 4.1.2]{HebeyBook1997}).
	Then, for every $u\in\mathcal{C}^\infty(M)$,
	\[
	\Vert \gamma u\Vert_{P_g}^2 =\int_M (u\circ\gamma^{-1})P_g(u\circ\gamma^{-1}) dV_g =\int_M (u\circ\gamma^{-1})P_g(u)\circ\gamma^{-1} \,dV_g = \int_M uP_gu \,dV_g =\Vert u\Vert^2_{P_g}.
	\]
Density of $\mathcal{C}^\infty(M)$ in $H_g^m(M)$ yields the result.
\end{proof}
\bigskip

Every $\gamma\in \text{Isom}(M,g)$ also induces a linear map $\gamma:L^p_g(M)\rightarrow L^p_g(M)$, $p\geq 1$, given by $u\mapsto u\circ\gamma^{-1}$, which is also an isometry, thanks to \eqref{Eq:Integral invariance isometries}.

\bigskip

Let $\Gamma$ be any closed subgroup of Isom$(M,g)$ such that $\dim\Gamma x\leq N-1$ for any $x\in M$. From now on, suppose that $\overline{\Omega}$ is $\Gamma$-invariant, namely, if $x\in\overline{\Omega}$, then $\Gamma x\subset\overline{\Omega}$. In this way, for any $u\in \mathcal{C}_c^\infty(\Omega)$ and every isometry $\gamma\in\Gamma$, it follows that
\[
\int_\Omega  u\circ\gamma\; dV_g = \int_{\Omega}  u\; dV_g
\]
and $\gamma$ also induces  linear isometries
\begin{equation}\label{Eq:LinearIsometryGamma}
    \gamma:H_{0,g}^m(\Omega)\rightarrow H_{0,g}^m(\Omega), \quad \text{ and }\quad \gamma:L_g^p(\Omega)\rightarrow L_g^p(\Omega).
\end{equation}
for any $p\geq 1$.

%We say that $u:{\Omega}\rightarrow\mathbb{R}$ is {\em $\Gamma$-invariant} if $u(\gamma(x))=u(x)$ for every $\gamma\in\Gamma$ and every $x\in \Omega$.
We define the Sobolev space of $\Gamma$-invariant functions as
\[
H_{0,g}^m(\Omega)^\Gamma:=\{ u\in H_{0,g}^m(\Omega) \;:\; u \text{ is }\Gamma\text{-invariant} \}.
\]
This is a closed subspace of $H_{0,g}^m(\Omega)$. In fact, if $\mathcal{C}_c^\infty(\Omega)^\Gamma$ denotes the space of smooth functions with compact support in $\Omega$ which are $\Gamma$-invariant, then $H_{0,g}^m(\Omega)$ coincides with the closure of this space under the Sobolev norm $\Vert \cdot \Vert_{P_g}$. As the dimension of any $\Gamma$ orbit is strictly less than $N$, the space $H_{0,g}^m(\Omega)^\Gamma$ is infinite dimensional, thanks to the existence of $\Gamma$-invariant partitions of the unity (Cf. \cite[Theorem 4.3.1]{Palais1961} and also \cite[Claim 3.66]{AlexBettiol}).

We will need the following Sobolev embedding result. 

\begin{lemma}\label{Lemma:Sobolev}
	Let $\Gamma$ be a closed subgroup of $\text{Isom}(M,g)$, $\kappa:=\min\{\dim\Gamma x\; : \; x\in M\}$, and let $\Omega$ be a $\Gamma$-invariant domain. Define 
	\[
	2^\ast_{m,\Gamma} :=
	\begin{cases}
		\frac{2(N-\kappa)}{(N-\kappa)-2m}, & N-\kappa>2m,\\
		\infty, & N-\kappa\leq 2m.
	\end{cases}
%\begin{tabular}{cc}
%		$\frac{2(N-d)}{(N-d)-2m}$, & $N-d>2m$\\
%		$\infty$, & $N-d\leq 2m$.
%	\end{tabular}\right.
	\]
	Then
	\[
	H_{0,g}^m(\Omega)^\Gamma\hookrightarrow L_g^r(\Omega)
	\]
	is continuous and compact for every $1\leq r < 2^\ast_{m,\Gamma}$. 
\end{lemma}

\begin{proof}
The case for $m=1$ is just Theorem 2.4 in  \cite{IvanovNazarov2007}.The case for $m>1$ follows from a bootstrap argument as in Proposition 2.11 in \cite{AubinBook}.
\end{proof}

In what follows, we will suppose that $\Gamma$ satisfies condition (\hyperref[Gamma:DimensionOrbits]{$\Gamma2$}). Under this hypothesis, the existence of least energy $\Gamma$-invariant solutions to  \eqref{eq:dirichlet} follows directly from standard variational methods using the Palais' Principle of Symmetric Criticality \cite{Palais1979} together with Lemma \ref{Lemma:Sobolev}. For the reader's convenience, we sketch the proof of a slightly more general result, namely, we  show the existence of $\Gamma$-invariant solutions to the problem
\begin{equation} \label{Problem:DirichletBoundary}
	\begin{cases}
		P_g u = |u|^{p-2}u, & \text{ in } \Omega,\\
		u\in H_{0,g}^m(\Omega)^\Gamma,
	\end{cases}
\end{equation}
where $p\in (2,2_m^\ast]$.

Consider the functional
\[
J_\Omega:H_{0,g}^m(\Omega)\rightarrow\mathbb{R}, \qquad J_\Omega(u):=\frac{1}{2}\Vert u\Vert_{P_g}^2 - \frac{1}{p}\int_\Omega \vert u\vert^p\; dV_g.
\]

Given that $\Vert \cdot\Vert_{P_g}$ is a well defined norm equivalent to the standard norm $\Vert \cdot \Vert_{H^m(\Omega)}$, Sobolev inequality yields that it is a $C^1$ functional for any $p\in(2,2_m^\ast]$. From \eqref{Eq:LinearIsometryGamma}, this functional is $\Gamma$-invariant, namely it satisfies that
\[
J_\Omega(u\circ\gamma^{-1})=J_\Omega(u).
\]
Hence, due to Palais' Principle of Symmetric Criticality \cite{Palais1979}, the critical points of $J_\Omega$ restricted to $H_{0,g}^m(\Omega)^\Gamma$ correspond to the $\Gamma$-invariant solutions to the problem \eqref{Problem:DirichletBoundary}. The nontrivial ones belong to the set
\[
\mathcal{M}_\Omega^\Gamma := \{u\in H_{0,g}^m(\Omega) \;:\; u\neq 0, J'_{\Omega}(u)u=0 \},
\]
which is a $C^1$ codimension one Hilbert manifold in $H_{0,g}^m(\Omega)^\Gamma$. Notice that
\[
J_{\Omega}(u) = \frac{m}{N}\Vert u\Vert_{P_g},\qquad u\in \mathcal{M}_\Omega^\Gamma.
\]
Thanks to the Sobolev inequalities \cite[Theorem 2.30]{AubinBook}, $\mathcal{M}_\Omega^\Gamma$ is closed and
\[
0<c_\Omega^\Gamma=\inf_{u\in\mathcal{M}_\Omega^\Gamma} J_\Omega(u).
\]

We say that $J_\Omega$ satisfies condition $(PS)_c^\Gamma$ in $H_{0,g}^m(\Omega)^\Gamma$ if every sequence $u_k\in H_{0,g}^m(\Omega)^\Gamma$ such that $J_\Omega(u_k)\rightarrow c$ and $\nabla J_{\Omega}(u_k)\rightarrow 0$ in $H_{0,g}^m(\Omega)^\Gamma$ as $k\rightarrow\infty$, has a convergent subsequence.

As $\Gamma$ satisfies (\hyperref[Gamma:DimensionOrbits]{$\Gamma2$}), then $\kappa\geq 1$ and $p\leq 2_m^\ast < 2_{m,\Gamma}^\ast$. Hence, by Lemma \ref{Lemma:Sobolev}, $J_\Omega$ satisfies condition $(PS)_{c_\Omega^\Gamma}^\Gamma$ and Theorem 7.12 in \cite{AmbrosettiMalchiodiBook} yields that $c_\Omega^\Gamma$ is attained. Thus, there exists a least energy $\Gamma$-invariant solution to the problem \eqref{Problem:DirichletBoundary}.

We summarize this analysis in the following proposition.

\begin{prop}\label{Proposition:ExistenceLeastEnergySolution}
If $\Gamma$ satisfies \emph{(\hyperref[Gamma:DimensionOrbits]{$\Gamma2$})} and if $P_g$ is coercive, then, for any $p\in(2,2_m^\ast]$ the problem \eqref{Problem:DirichletBoundary} admits a least energy $\Gamma$-invariant solution.
\end{prop}

To our knowledge, this is the first existence result of symmetric least energy solutions to the homogeneous Dirichlet boundary problem \eqref{Problem:DirichletBoundary}. Another result for non-homogeneous Dirichlet boundary conditions to problems involving the $GJMS$-operators, in the absence of symmetries, can be found in \cite{BekiriBenalili2018,BekiriBenalili2019,BekiriBenalili2022}.

\begin{remark}
With slight modifications, the same result is true for operators given by a linear combination of Laplacians, i.e., for operators having the form
\begin{equation*}
%\label{Operator:Sum of Laplacians}
\hat{P}_g = 
\begin{cases}
\sum_{j=0}^{m/2}a_i(-\Delta)^{j}, & m \text{ even}, \\
\sum_{j=0}^{(m-1)/2}a_i(-\Delta)^{j}, & m \text{ odd},
\end{cases}
\end{equation*}
where $a_0\in \mathcal{C}^\infty(M)^\Gamma$ is positive, $a_{m/2}>0$ and $a_j\geq 0$ are constants for $j=1,\ldots,m-1$ if $m$ is even, and $a_{(m-1)/2}>0$ and $a_j\geq 0$ are constants for $j=1,\ldots,(m-3)/2$ if $m$ is odd. This is true because the Principle of Symmetric Criticality can be applied by noticing that
\[
\Delta^{i+1}(u\circ\gamma) =  \Delta^{i+1}(u)\circ\gamma
\]
for any $u\in \mathcal{C}_c^\infty(\Omega)^\Gamma$ and every isometry $\gamma\in\Gamma$ (see, for instance, \cite[Remark 6.9 (c)]{AmannEscherBook}).\qed

\end{remark}

%%%%%%%%%%%%%%%%%%%%%
%%%%%%%%%%%%%%%%%%%%

\section{The polyharmonic system}
\label{sec:system}

We next study the system \eqref{Eq:Q Systems}.   Fix $\ell\in\mathbb{N}$ and consider the product space $\left( H_g^m(M) \right)^\ell$
endowed with the norm
\begin{equation*}
%\label{Eq:NormProduct}
	\Vert \overline{u}\Vert :=\Vert(u_1,\ldots,u_\ell)\Vert := \Big(\sum_{i=1}^\ell \Vert u_i \Vert_{P_g}^2\Big)^{1/2}.
\end{equation*}
Let $\mathcal{J}:\left( H_g^m(M) \right)^\ell\rightarrow\R$ be the functional given by
\[
\mathcal{J}(\overline{u}):=\frac{1}{2}\sum_{i=1}^\ell\Vert u_i\Vert_{P_g}^2 - \frac{1}{2^*_m}\sum_{i=1}^\ell\int_{M}\nu_i\vert u_i\vert^{2^*_m} - \frac{1}{2}\sum_{\substack{i,j=1 \\ j\neq i}}^\ell\int_{M}\eta_{ij}\vert u_j\vert^{\alpha_{ij}}\vert u_i\vert^{\beta_{ij}}.
\]
This is a $C^1$ functional and its partial derivatives are given by
\begin{align*}
	&\partial_i\mathcal{J}(\overline u)v_i=\langle u_i, v_i\rangle_{P_g} - \int_{M} \nu_i|u_i|^{2_m^*-2}u_iv_i - \sum_{\substack{j=1 \\ j\neq i}}^\ell\int_{M}\eta_{ij}\beta_{ij}|u_j|^{\alpha_{ij}}|u_i|^{\beta_{ij}-2}u_iv_i, 
\end{align*}
for every $\overline{v}\in \left( H_g^m(M) \right)^\ell$ and every $i=1,\ldots,\ell.$
Hence, solutions to the system \eqref{Eq:Q Systems} correspond to the critical points of $\mathcal{J}.$

Fix now a closed subgroup $\Gamma$ of isometries satisfying (\hyperref[Gamma:DimensionOrbits]{$\Gamma2$}) and define $\cH:=(H^m_g(M)^\Gamma)^\ell$. This is a closed subspace of $\left(H_g^1(M)\right)^\ell$.  By Lemma \ref{Lemma:Pg symmetry invariance}, $\mathcal{J}$ is $\Gamma$-invariant and by the Principle of Symmetric Criticality \cite{Palais1979}, the critical points of $\mathcal{J}$ restricted to $\mathcal{H}$ are the $\Gamma$-invariant solutions to the system \eqref{Eq:Q Systems}. Hence, we can restrict ourselves to seek critical points of $\mathcal{J}$ in $\mathcal{H}$. Observe that the fully nontrivial ones belong to the set
\[
\mathcal{N}:=\{\overline{u}\in\mathcal{H}\;:\; u_i\neq 0,\ \partial_i\mathcal{J}(\overline{u})u_i = 0, \text{ for each }i=1,\ldots,\ell\}.
\]

It is readily seen that 
\begin{equation} \label{eq:energy_nehari}
	\cJ(\overline{u})=\frac{m}{N}\|\overline{u}\|^2,\qquad\text{if \ }\overline{u}\in\cN.
\end{equation}
\medskip

\begin{lemma} \label{lem:away_froM_{d/2}}
	There exists $d_0>0$, independent of $\eta_{ij}$, such that $\min_{i=1,\ldots,\ell}\|u_i\|\geq d_0$ if $\overline{u}=(u_1,\ldots,u_\ell)\in \cN$. Thus, $\cN$ is a closed subset of $\cH$ and $\inf_\cN\cJ>0$.
\end{lemma}

\begin{proof}
	Since $\eta_{ij}<0$ and as the norm $\Vert\cdot\Vert_{P_g}$ is equivalent to the standard norm in $H_g^m(M)^\Gamma$, for any $\overline{u}\in\mathcal{N}$, it follows from the Sobolev inequality the existence of a constant $C>0$ such that
	\begin{align*}
		\|u_i\|_{P_g}^2\leq \int_{M} \nu_i|u_i|^{2_m^*}\leq C\|u_i\|_{P_g}^{2_m^*}\quad \text{ for \ }\overline{u}\in \cN, \ i=1,\ldots,\ell.
	\end{align*}
	The result follows from this inequality.
\end{proof}
\medskip

A fully nontrivial solution $\overline{u}$ to the system \eqref{Eq:Q Systems} satisfying $\cJ(\overline{u})=\inf_\cN\cJ$ is called a \emph{$\Gamma$-invariant least energy solution}. To establish the existence of fully nontrivial critical points of $\cJ$, we follow the variational approach introduced in \cite{ClappSzulkin19}. The proof of Theorem \ref{Th:MainQSystems} is, up to minor modifications, the same as in \cite[Theorem 1.1]{ClappFernandezSaldana2021}, but we sketch it for the reader's convenience.

Given $\overline{u}=(u_1,\ldots,u_\ell)$ and $\overline{s}=(s_1,\ldots,s_\ell)\in(0,\infty)^\ell$, we write
\[
\overline{s}\,\overline{u}:= (s_1u_1,\ldots,s_\ell u_\ell).
\]
Let $\mathcal S:=\{u\in H_g^m(M)^\Gamma:\|u\|=1\}$, define  $\cT:=\mathcal S^\ell$ and 
$$\cU:=\{\overline{u}\in\cT:\overline{s}\,\overline{u}\in\cN\text{ \ for some \ }\overline{s}\in(0,\infty)^\ell\}.$$

The next result is proved exactly in the same way as \cite[Proposition 3.1]{ClappSzulkin19}.

\begin{lemma} \label{lem:U}
	\begin{itemize}
		\item[$(i)$] Let $\overline{u}\in\cT$. If there exists $\overline{s}_{\overline{u}}\in(0,\infty)^\ell$ such that $\overline{s}_{\overline{u}}\overline{u}\in\cN$, then $\overline{s}_{\overline{u}}$ is unique and satisfies
		$$\cJ(\overline{s}_{\overline{u}}\overline{u})=\max_{\overline{s}\in(0,\infty)^\ell}\cJ(\overline{s}\,\overline{u}).$$
		\item[$(ii)$] $\cU$ is a nonempty open subset of $\cT$, and the map $\cU\to(0,\infty)^\ell$ given by $\overline{u}\mapsto\overline{s}_{\overline{u}}$ is continuous.
		\item[$(iii)$] The map $\rho:\cU\to \cN$ given by $\overline{u}\mapsto\overline{s}_{\overline{u}}\overline{u}$ is a homeomorphism.
		\item[$(iv)$] If $(\overline{u}_n)$ is a sequence in $\cU$ and $\overline{u}_n\to\overline{u}\in\partial\cU$, then $|\overline{s}_{\overline{u}_n}|\to\infty$.
	\end{itemize}
\end{lemma}
\medskip

Define $\Psi:\cU\to\r$ as 
\begin{equation*}
	\Psi(\overline{u}): = \cJ(\overline{s}_{\overline{u}}\overline{u}).
\end{equation*}
According to Lemma \ref{lem:U}, $\cU$ is a Hilbert manifold, for it is an open subset of the smooth Hilbert submanifold $\cT$ of $\cH$. When $\Psi$ is differentiable at $\overline{u}$, we write $\|\Psi'(\overline{u})\|_*$ for the norm of $\Psi'(\overline{u})$ in the cotangent space $\mathrm{T}_{\overline{u}}^*(\cT)$ to $\cT$ at $\overline{u}$, i.e.,
$$\|\Psi'(\overline{u})\|_*:=\sup\limits_{\substack{\overline{v}\in\mathrm{T}_{\overline{u}}(\cU) \\\overline{v}\neq 0}}\frac{|\Psi'(\overline{u})\overline{v}|}{\|\overline{v}\|},$$
where $\mathrm{T}_{\overline{u}}(\cU)$ is the tangent space to $\cU$ at $\overline{u}$.

Recall that a sequence $(\overline{u}_n)$ in $\cU$ is called a $(PS)_c$\emph{-sequence for} $\Psi$ if $\Psi(\overline{u}_n)\to c$ and $\|\Psi'(\overline{u}_n)\|_*\to 0$, and $\Psi$ is said to satisfy the $(PS)_c$\emph{-condition} if every such sequence has a convergent subsequence. Similarly, a $(PS)_c$\emph{-sequence for} $\cJ$ is a sequence $(\overline{u}_n)$ in $\cH$ such that $\cJ(\overline{u}_n)\to 0$ and $\|\cJ'(\overline{u}_n)\|_{\cH'}\to 0$, and $\cJ$ satisfies the $(PS)_c$\emph{-condition} if any such sequence has a convergent subsequence.   Here $\cH'$ denotes the dual space of $\cH$.

\begin{lemma}\label{Lemma:Psi}
	\begin{itemize}
		\item[$(i)$] $\Psi\in\cC^1(\cU)$ and its derivative is given by
		\begin{equation*}
			\Psi'(\overline{u})\overline{v} = \cJ'(\overline{s}_{\overline{u}}\overline{u})[\overline{s}_{\overline{u}}\overline{v}] \quad \text{for all } \overline{u}\in\cU \text{ and }\overline{v}\in \mathrm{T}_{\overline{u}}(\cU).
		\end{equation*}
		Moreover, there exists $d_0>0$ such that
		$$d_0\min_i\{s_{u,i}\}  \|\cJ'(\overline{s}_{\overline{u}}\overline{u})\|_{\cH'}\leq\|\Psi'(\overline{u})\|_*\leq \max_i\{s_{u,i}\} \|\cJ'(\overline{s}_{\overline{u}}\overline{u})\|_{\cH'}\quad \text{for all } \overline{u}\in\cU.$$
		%where $|\overline{s}|_\infty=\max\{|s_1|,\ldots,|s_\ell|\}$ if $\overline{s}=(s_1,\ldots,s_\ell)$.
		\item[$(ii)$] If $(\overline{u}_n)$ is a $(PS)_c$-sequence for $\Psi$ in $\mathcal{U}$, then $(\overline{s}_{\overline{u}_n}\overline{u}_n)$ is a $(PS)_c$-sequence for $\cJ$ in $\mathcal{H}$.
		\item[$(iii)$] $\overline{u}$ is a critical point of $\Psi$ if and only if $\overline{s}_{\overline{u}}\overline{u}$ is a fully nontrivial critical point of $\cJ$.
		\item[$(iv)$] If $(\overline{u}_n)$ is a sequence in $\cU$ and $\overline{u}_n\to\overline{u}\in\partial\cU$, then $|\Psi(\overline{u}_n)|\to\infty$.
		\item[$(v)$]$\overline{u}\in\cU$ if and only if $-\overline{u}\in\cU$, and $\Psi(\overline{u})=\Psi(-\overline{u})$.
	\end{itemize}
\end{lemma}

We omit the proof of this lemma, because the argument is exactly the same as in \cite[Theorem 3.3]{ClappSzulkin19}.

\begin{lemma}
For every $c\in\r$,	$\Psi$ satisfies the $(PS)_c$-condition.
\end{lemma}

\begin{proof}
	First observe that a $(PS)_c$-sequence $(\overline{v}_n)$ for $\mathcal{J}$ is bounded. Indeed, there exists $C>0$ such that
	\[
	\frac{m}{N}\Vert \overline{v}_n \Vert^2 = \mathcal{J}(\overline{v}_n) - \frac{1}{2_m^\ast} \mathcal{J}'(\overline{v}_n)\overline{v}_n \leq C(1 + \Vert \overline{v}_n\Vert),
	\]
	and the claim follows.
	
	Using this, let $(\overline{u}_n)\subset\mathcal{U}$ be a $(PS)_c$-sequence for $\Psi$. By Lemma \ref{Lemma:Psi}, the sequence $\overline{v}_n:=\rho(\overline{u})\in\mathcal{N}$ is a $(PS)_c$-sequence for $\mathcal{J}$ and it is bounded by the above claim. A standard argument using Lemma \ref{Lemma:Sobolev} as in \cite[Proposition 3.6]{ClappPistoia2018}, shows that $(\overline{v}_n)$ contains a convergent subsequence, converging to some  $\overline{v}\in\mathcal{H}$. As $\overline{v}_n\in\mathcal{N}$ for every $n\in\mathbb{N}$ and as $\mathcal{N}$ is closed by Lemma \ref{lem:away_froM_{d/2}}, it follows that $\overline{v}\in\mathcal{N}$. Finally, since $\rho$ is a homeomorphism between $\mathcal{N}$ and $\mathcal{U}$, this yields that $\overline{u}_n$ converges to $\rho^{-1}(\overline{v})$ in a subsequence, and $\Psi$ satisfies the $(PS)_c$-condition
\end{proof}

Given a nonempty subset $\mathcal{Z}$ of $\mathcal{T}$ such that $\overline{u}\in\mathcal{Z}$ if and only if $-\overline{u}\in\mathcal{Z}$, the \emph{genus of $\mathcal{Z}$}, denoted by $\mathrm{genus}(\mathcal{Z})$, is the smallest integer $k\geq 1$ such that there exists an odd continuous function $\mathcal{Z}\rightarrow\mathbb{S}^{k-1}$ into the unit sphere $\mathbb{S}^{k-1}$ in $\R^k$. If no such $k$ exists, we define $\mathrm{genus}(\mathcal{Z})=\infty$; finally, we set $\mathrm{genus}(\emptyset)=0$.

\begin{lemma}
	$\mathrm{genus}(\cU)=\infty$.
\end{lemma}

\begin{proof}
	Condition (\hyperref[Gamma:DimensionOrbits]{$\Gamma2$}) together with the existence of $\Gamma$-invariant partitions of the unity (see \cite{Palais1961}), one obtains an arbitrarily large number of positive $\Gamma$-invariant functions in $\cC^\infty(M)$ with mutually disjoint supports. Then, arguing as in \cite[ Lemma 4.5]{ClappSzulkin19}, one shows that $\mathrm{genus}(\cU)=\infty$.
\end{proof}
\smallskip

\begin{proof}[Proof of Theorem \ref{Th:MainQSystems}]
	Lemma \ref{Lemma:Psi} $(iv)$ implies that $\cU$ is positively invariant under the negative pseudogradient flow of $\Psi$, so the usual deformation lemma holds true for $\Psi$, see e.g. \cite[Section II.3]{StruweBook} or \cite[Section 5.3]{WillemBook}. As $\Psi$ satisfies the $(PS)_c$-condition for every $c\in\r$, standard variational arguments show that $\Psi$ attains its minimum on $\cU$ at some $\overline{u}$. By Lemma \ref{Lemma:Psi}$(iii)$ and the Principle of Symmetric Criticality, $\overline{s}_{\overline{u}}\overline{u}$ is a $\Gamma$-invariant least energy fully nontrivial solution for the system \eqref{Eq:Q Systems}. Moreover, as $\Psi$ is even and $\mathrm{genus}(\cU)=\infty$, arguing as in the proof of Theorem 3.4 (c) in \cite{ClappSzulkin19}, it follows that  $\Psi$ has an unbounded sequence of critical points. Using Lemma \ref{Lemma:Psi} (iii), and the fact that $\Psi(\overline{u})=\cJ(\overline{s}_{\overline{u}}\overline{u})=\frac{m}{N}\|\overline{s}_{\overline{u}}\overline{u}\|^2$ by \eqref{eq:energy_nehari}, the system \eqref{Eq:Q Systems} has an unbounded sequence of fully nontrivial $\Gamma$-invariant solutions.
\end{proof}
\medskip

We next apply Theorem \ref{Th:MainQSystems} to the case $\ell=1$ and a recent result by J. Vétois to prove the multiplicity result stated in Corollary \ref{Corollary:NodalSolutionsPaneitz}.

\begin{proof}[Proof of Corollary \ref{Corollary:NodalSolutionsPaneitz}]
 Theorem 2.2 in \cite{Vetois2022} states that the positive solutions to the problem \eqref{Equation:Paneitz-Branson} must be constant and by the concrete expression of the Paneitz-Branson operator and the $Q$-curvature on Einstein manifolds (see, for instance, \cite{DjadliHebeyLedoux2000} and \cite{Gover06} respectively) it is unique. As $(M,g)$ is Einstein with positive scalar curvature, the operator $P_g$ is coercive \cite[Proposition 4]{Ro} and Theorem \ref{Th:MainQSystems} for $\ell=1$, yields the existence of an unbounded sequence of $\Gamma$-invariant solutions, and the corollary follows.  
\end{proof}

%%%%%%%%%%%%%%%%%%%
%%%%%%%%%%%%%%%%%%%%
\section{One dimensional reduction}\label{Section:OneDimensionalReduction}

In this section, we will strongly use that the group $\Gamma$ satisfies properties (\hyperref[Gamma:Cohomogeneity]{$\Gamma1$}) and (\hyperref[Gamma:MetricDecomposition]{$\Gamma3$}). Recall that $M_-$ and $M_+$ denote the singular orbits, as it was given in the introduction, and let $n_1=\dim M_{-}$ and $n_2=\dim M_{+}$, $N-n_{i}\geq 2$. Since $M$ is compact, the geodesic distance between $M_-$ and $M_+$,
\[
d:=\text{dist}_g(M_-,M_+),
\]
is attained and the distance function $r:M\rightarrow[0,d]$ given by
\[
r(x):=\text{dist}_g(M_-,x),
\]
is well defined. This function is a Riemannian submersion and satisfies for any $x\in M\setminus(M_+\cup M_-)$ that
\[
\vert \nabla r(x)\vert_g^2 = 1 \quad \text{and} \quad \Delta_g r(x) = h(r(x)),
\]
where $h(t)$ denotes the mean curvature of $r^{-1}(t)$. See \cite[Chapter 2, Section 4.1]{Petersen06}, and also \cite[Section 2]{BBP21}. For the mean curvature, we explicitly have:
\[
h(t):=\frac{\text{codim}(M_-)-1}{t} + t(\text{trace}(A)) + \text{trace}(B) + o(t^2),
\]
for some matrices $A$ and $B$ not depending on $t$, and it further satisfies that
\[
\lim_{t\rightarrow 0} t\cdot h(t) = N - n_1-1 \quad\text{and}\quad \lim_{t\rightarrow d}(t-d)h(t) = N-n_2 - 1
\]
(see \cite{GeTang2014}, and also \cite{BBP21}). Therefore, for any $w\in \mathcal{C}^\infty([0,d])$ the following identity holds true:
\begin{equation}\label{Equation:IsoparametricLaplacian}
    \Delta_g(w\circ r) =
    \begin{cases}
    (N-n_1) w''(0), & \text{in } M_-,\\
    (w'' + hw')\circ r, & \text{in } M\setminus(M_-\cup M_+),\\
    (N-n_2)w''(d), & \text{in } M_+.
    \end{cases}
\end{equation}

Moreover, notice that $M_t:=r^{-1}(t)$ is a principal orbit for any $t\in(0,d)$, while $M_-=r^{-1}(0)$ and $M_+=r^{-1}(d)$. Hence, for every $x,y\in M$, it follows that
\[
r(x)=r(y)\Longleftrightarrow x,y \in M_t \text{ for some }t\in[0,d]\Longleftrightarrow \Gamma x = \Gamma y.
\]
Thus, for any $w\in \mathcal{C}^\infty([0,d])$, the function $w\circ r\in \mathcal{C}^\infty(M)^\Gamma$ and, conversely, for any $u\in \mathcal{C}^\infty(M)^\Gamma$ there exists a unique $w\in \mathcal{C}^\infty([0,d])$ such that $u=w\circ r$. In this way, we have a linear isomorphism 
\begin{equation}\label{Eq:IsomorphismSmoothFunctionSpace}
\iota: C^\infty(M)^\Gamma \rightarrow C^\infty([0,d]),\qquad u=w\circ r \mapsto w.
\end{equation}

As in the introduction, $K$ will denote the principal isotropy, i.e., the stabilizer of the $\Gamma$-action at any point $p_0\in r^{-1}(t_0)$, for some $t_0\in(0, d)$. Such a group is the same at the preimage of any interior point of $(0, d)$ under $r$. All the regular orbits are diffeomorphic to $\Gamma/K$, and we will fix one, say $M_{d/2}:=r^{-1}(d/2)$.

\begin{lemma}\label{Lemma:MetricVolumeDecomposition}
Assume \emph{(\hyperref[Gamma:Cohomogeneity]{$\Gamma1$})} and \emph{(\hyperref[Gamma:MetricDecomposition]{$\Gamma3$})} hold true. Then there exists a metric $g_\ast$ on $M_{d/2}$, a diffeomorphism $ \varphi:(0,d)\times M_{d/2}\rightarrow M\setminus(M_-\cup M_+)$ and a smooth function $\phi:[0,d]\rightarrow\mathbb{R}$ such that

% Then there exist metrics $g_\ast$ and ${g}_{t\in[0,d]}$ on $M_{d/2}$, a diffeomorphism $ \varphi:(0,d)\times M_{d/2}\rightarrow M\setminus(M_-\cup M_+)$ and a smooth function $\phi:[0,d]\rightarrow\mathbb{R}$ such that
\begin{enumerate}
    % \item ${g}_{t}$, $t\in(0,d)$ is a smooth one parameter family of metrics on $M_{d/2}$ with the appropiate smoothness conditions as $t\rightarrow 0$ and $t\rightarrow d$, and 
    % \[
    % g = dt^2 + g_t, t\in(0,d). 
    % \]
        \item for every $(x,t)\in M_{d/2}\times(0,d)$, $r\circ\varphi(x,t) = t$.
    \item $dV_g = \phi(t) dt\wedge dV_{g_\ast}$. 
\end{enumerate}
\end{lemma}

\begin{proof}
% The third item is a direct consequence of property follow from the previous discussion. The metric $g_\ast$ is given by the pullback of $B$ under the identification $\Gamma/K\times (0, d)\simeq M\setminus(M_{-}\cup M_{+})$. It remains to prove (2).

The first item only depends on (\hyperref[Gamma:Cohomogeneity]{$\Gamma1$}) as follows: For any $x\in M_{d/2}$, consider the unique minimizing horizontal geodesic  $c:[0, d]\to M$ joining $M_-$ and $M_+$ such that $c(d/2)=x$ (Cf. \cite[Proposition 3.78]{AlexBettiol}). Then, the diffeomorphism $\varphi$ is given by
\[
\begin{split}
\varphi: (0, d)\times M_{d/2}&\to M\setminus(M_-\cup M_+)\\
(t, x)&\mapsto c(t).
\end{split}
\]
If necessary, we may reparametrize $c$ so that $r\circ \varphi(t, x)=t$. 
\medskip

The second item follows from the fact that the volume form of a metric given as in (\hyperref[Gamma:MetricDecomposition]{$\Gamma3$}), is the volume product. That is, for a local coordinate system $(t, x^1,\dots, x^{N-1})$ in $M$, around an arbitrary point $x\in M$, the set $\left\{\frac{\partial }{\partial t}, \frac{\partial }{\partial x^{1}}, \dots, \frac{\partial }{\partial x^{N-1}} \right\}$
is a basis of the tangent space $T_xM$. Then, around $x$ the metric $g$ is given by the matrix 
\[
[g]=
\left[\begin{array}{ccccc}
1 &0 &0 &\cdots & 0\\
0 & f_1^{2}(t) [g_1] &0 &\cdots & 0\\
\vdots &\vdots & \vdots & \ddots &\vdots\\
0 & 0 & 0 &\cdots & f_k^2(t) [g_k]
\end{array}\right],
\]
where $[g_j]$ is the matrix  corresponding to the metric $g_j$, $j=1, \dots, k$. If $d_j\times d_j$ is the size of $[g_j]$, then $\sum d_j= N-1$, and the volume form of $g$ is given by 
\begin{eqnarray*}
dV_g&=&\sqrt{{\rm det}([g])}\ dt\wedge dx^1\wedge\cdots\wedge dx^{N-1}\\
&=& \prod_{j=1}^kf^{d_j}(t) \sqrt{{\rm det}([g_j])}\  dt \wedge dx^1\wedge\cdots\wedge dx^{N-1}.
\end{eqnarray*}
Define $\phi(t):={\prod_{j=1}^k} f^{d_j}(t)$ and take $g_\ast:={\sum_{i=1}^k} g_i$, which is a metric on $M_{d/2}$.
 Therefore, $dV_{g_\ast}$ is given by
\[
dV_{g_\ast}=\prod_{j=1}^k\sqrt{{\rm det}([g_j])}\  dx^1\wedge\cdots\wedge dx^{N-1},
\]
and we conclude the result.
\end{proof}

% \begin{lemma}\label{Lemma:VolumeFunction}
% The function $\beta(t):=\text{Vol}(M_{d/2},g_t)$ can be extended smoothly to the whole interval $[0,d]$ and 
% \[
% \textcolor{red}{\lim_{t\rightarrow 0} \beta(t) = 0? \quad\text{and}\quad \lim_{t\rightarrow d} \beta(t) = 0?}
% \]
% \end{lemma}

% \begin{proof}
% \textcolor{red}{Insertar prueba o referencia aquí}

We can adapt Lemma 2.2 in \cite{FdzPetean20} to the context of cohomogeneity one actions. Recall that $g_t=\sum f_j^2(t)g_j$ denotes the metric given to the principal orbits in \emph{(\hyperref[Gamma:MetricDecomposition]{$\Gamma3$})}.

\begin{lemma}\label{Lemma:IntegralFormula}
For any integrable function $\psi:[0,d]\rightarrow \mathbb{R}$,
\[
\int_{M}\psi\circ r \; dV_g = \int_0^d \text{Vol}(M_{d/2},g_t)\psi(t)\, dt = \text{Vol}(M_{d/2},g_\ast)\int_0^d \psi(t)\phi(t)\, dt.
\]
In particular,
\[
\text{Vol}(M_{d/2},g_t) = \text{Vol}(M_{d/2},g_\ast)\phi(t).
\]
\end{lemma}

\begin{proof}
As $M_+\cup M_-$ has Lebesgue measure zero on $M$, by Lemma \ref{Lemma:MetricVolumeDecomposition} and Fubini's Theorem, we obtain on the one hand that
\begin{align*}
\int_M \psi\circ r \; dV_g = \int_{M\setminus(M_-\cup M_+)}\psi\circ r\; dV_g &= \int_{M_{d/2}\times(0,d)}\psi\circ(r\circ\varphi)(t,x)\; dt\wedge dV_{g_t}\\
&=\int_0^d \int_{M_{d/2}}\psi(t) \; dV_{g_t} \; dt\\
&=\int_{0}^d \psi(t) \int_{M_{d/2}} \; dV_{g_t}\, dt\\
&=\int_{0}^d \psi(t) \text{Vol}(M_{d/2},g_t) \,dt.
\end{align*}
On the other hand, using the second expression for the volume element in $M_{d/2}\times(0,d)$ in \ref{Lemma:MetricVolumeDecomposition},

\begin{align*}
\int_M \psi\circ r \; dV_g = \int_{M\setminus(M_-\cup M_+)}\psi\circ r\; dV_g &= \int_{M_{d/2}\times(0,d)}\psi\circ(r\circ\varphi)(t,x)\; \phi(t)dt\wedge dV_{g_\ast}\\
&=\int_0^d \int_{M_{d/2}}\psi(t) \phi(t)\; dV_{g_\ast} \; dt\\
&= \text{Vol}(M_{d/2},g_\ast)\int_0^d \psi(t)\phi(t)\, dt
\end{align*}
where we conclude the integral identity.

For the volume identity, subtracting the above identities we obtain
\[
\int_0^d [\text{Vol}(M_{d/2},g_\ast)\phi(t) - \text{Vol}(M_{d/2},g_t) ]\psi(t) dt = 0,
\]
for any integrable function $\psi$. 
%Taking $\psi=\text{Vol}(M_{d/2},g_\ast)\phi(t) - \text{Vol}(M_{d/2},g_t)$, we obtain that
%\[
%\int_0^d [\text{Vol}(M_{d/2},g_\ast)\phi(t) - \text{Vol}(M_{d/2},g_t) ]^2 dt = 0,
%\]
%where 
 We conclude that $\text{Vol}(M_{d/2},g_\ast)\phi(t) = \text{Vol}(M_{d/2},g_t)$ almost everywhere in $[0,d]$. As the volume function and $\phi$ are continuous, we conclude the identity.
\end{proof}

Now we study the preimage of measure zero subsets in $[0,d]$ under the distance function $r$. To this end, denote the Lebesgue measure in $[0,d]$ by $\lambda$, and by $\lambda_g$ the Lebesgue measure in $(M,g)$.

\begin{lemma}\label{Lemma:MeasureZero}
If $E\subseteq[0,d]$ satisfies $\lambda(E)=0$, then $\lambda_g(r^{-1}(E))=0$.
\end{lemma}

\begin{proof}

Observe that the critical point set of $r$ is exactly the union of the singular orbits of the action, $M_{-}$ and $M_{+}$. Those points correspond to the endpoints of $[0, d]$, under $r$. By dimension reasons, $\lambda_g(M_{-}\cup M_{+})=0$. Therefore, it is enough to prove the statement for any proper subset $E\subset (0, d)$ of zero Lebesgue measure. Hence, we may assume that $\nabla r\neq 0$ at any point in $r^{-1}(E)$. Recall that, $r^{-1}(c)$ is a copy of the principal orbit, i.e., it is a submanifold of $M$ of dimension $n-1$, for any $c\in E$. Therefore $\lambda_g(r^{-1}(c))=0$, for any $c\in E$.
\newline

%Note that there exists a countable covering of $r^{-1}(E)$, formed by open sets where the implicit function theorem applies. Let $U$ be one of such sets, and w

Write $A:=r^{-1}(E)$. If $E$ is countable, then $A$ is a countable union of zero measure sets, so $A$ has measure zero. As $M_+\cup M_-$ has measure zero in $M$ and as $\vert\nabla r\vert_g=1$ in $M\smallsetminus(M_+\cup M_-)$, if $E$ is uncountable, the co-area formula yields that
\begin{eqnarray*}
\lambda_g(A)
%\int_{M}\chi_A\ dV_g&=&\int_M|\nabla r|_g\chi_A\ dV_g\\
&=&\int_{M\setminus (M_{-}\cup M_{+})}\chi_A\ dV_g\\
&=&\int_{(0, d)}\left[\int_{\{x\in A|\ r(x)=t\}}\chi_A\ dV_{g_t} \right]\ dt\\
&=&\int_{(0, d)}\left[\int_{r^{-1}(t)\cap A}\chi_A\ dV_{g_t} \right]\ dt\\
&=&\int_{E}\left[\int_{r^{-1}(t)\cap A}\ dV_{g_t} \right]\ dt
\end{eqnarray*}
where $\chi_A$ is the characteristic function of $A$ in $M$. Since $\lambda(E)=0$, then the Lebesgue integral of any measurable function over $E$ is zero. In particular, this implies that $\lambda_g(A)=0$.
\end{proof}
\medskip

In what follows, we will denote the set of positive and smooth $\Gamma$-invariant functions on $M$ by $\mathcal{C}_+^\infty(M)^\Gamma$ and by $\mathcal{C}^\infty_+([0,d])$ the set of positive and smooth functions in $[0,d]$. Next, we study how the symmetries allow us to reduce the operator $P_g$ into an operator acting on smooth functions defined in the interval $[0,d]$. In order to motivate a more general differential operator for which our theory holds true, first observe that if $(M,g)$ is Einstein with positive scalar curvature $\mu$, the higher order conformal operator, $P_g$, can be written as
\[
P_g = \prod_{i=1}^{m}\left( -\Delta_g + c_i \right)
\]
for some suitable constants $c_i>0$ (see \cite{Gover06, Juhl13}). On the other hand, in case $m=1$, when the scalar curvature $R_g$ is positive, the conformal Laplacian is simply 
\[
P_g=-\Delta_g + \frac{N-2}{4(N-1)}R_g=-\Delta_g + \frac{N-2}{4(N-1)}R_g(-\Delta_g)^0,
\] 
with $\frac{N-2}{4(N-1)}R_g\in C^\infty_+(M)^\Gamma$. Hence, for any $2m<N$ and any $\textbf{a}:=(a_0,a_1,\ldots,a_m)\in \mathcal{C}_+^\infty(M)^\Gamma\times(0,\infty)^{m}$, we are led to define the operator 
\begin{equation}\label{Eq:OperatorSumLaplacians}
P_{\textbf{a}}:=\sum_{i=0}^m a_i(-\Delta_g)^{i}.
\end{equation}

As for any $i\in\mathbb{N}\cup\{0\}$ and any pair of functions $u,v\in \mathcal{C}^\infty(M)$ we have that
\begin{equation}\label{Equation:PowerLaplacian}
\int_M v(-\Delta_g)^{i}u \; dV_g =
\begin{cases}
\int_M \Delta_g^{i/2} v\Delta_g^{i/2} u \; dV_g, & i \text{ even,} \\
\int_M \langle \nabla\Delta_g^{(i-1)/2}v, \nabla\Delta_g^{(i-1)/2} u\rangle_g \; dV_g, & i \text{ odd},
\end{cases}
\end{equation}
then the bilinear form 
\begin{equation}\label{Equation:InteriorProductOperator}
\begin{split}
&(u,v)_{g,\textbf{a}}:= \int_M vP_{\textbf{a}} u \; dV_g \\
&=\sum_{\substack{i=0\\ i\  even}}^m  \int_M a_i \Delta_g^{i/2} v\Delta_g^{i/2} u \; dV_g
+ \sum_{\substack{i=0\\ i\  odd}}^m \int_M a_i\langle \nabla\Delta_g^{(i-1)/2}v, \nabla\Delta_g^{(i-1)/2} u\rangle_g \; dV_g,\quad u,v\in \mathcal{C}^\infty(M)
\end{split}
\end{equation}
is positive definite and yields a norm $\Vert\cdot\Vert_{g,\textbf{a}}$ in $H_g^m(M)$, equivalent to the standard norm in $H_g^1(M)$. Note that the term for $i=0$ is simply $\int_M a_0 uv\; dV_g,$ and $a_0>0$ but not necessarily constant.

On the other hand, let $\alpha_0\in C^\infty_+([0,d])$ be such that $a_0=\alpha_0\circ r$,  $\beta(t):=\text{Vol}(M_{d/2},g_t)$, $t\in[0,d]$, and define the operator  $\mathcal{L}:\mathcal{C}^\infty(0,d)\rightarrow \mathcal{C}^\infty(0,d)$ by
\[
\mathcal{L} := \frac{d^2}{dt^2} + h(t) \frac{d}{dt}.
\]
For $w\in \mathcal{C}^\infty([0,d])$ define
\[
\Vert w\Vert_{\beta,\textbf{a}}:=\left( 
\sum_{\substack{i\neq0\\ i\  even}}^m a_i \int_0^d \vert \mathcal{L}^{i/2} w \vert^2 \beta\ dt 
+ \sum_{\substack{i=0\\ i\  odd}}^m a_i\int_0^d \vert \left( \mathcal{L}^{(i-1)/2} w \right)' \vert^2 \beta\ dt + \int_0^d \alpha_0 \vert w\vert^2 \beta\, dt
\right)^{1/2}
\]
where $\mathcal{L}^{i}$ denotes the $i$-fold composition of $\mathcal{L}$, $\mathcal{L}^0= Id$ and
\[
\left(  \mathcal{L}^{i}w \right)' := \frac{d}{dt}\left[\left( \frac{d^2}{dt^2} + h(t) \frac{d}{dt} \right)^{i} w\right].
\]

\begin{prop}\label{Proposition:EqualityNorms}
For any $\textbf{a}\in \mathcal{C}_+^\infty(M)^\Gamma\times(0,\infty)^{m}$ and any $u=w\circ r\in \mathcal{C}^\infty(M)^\Gamma$, 
\[
\Vert w\circ r\Vert_{g,\textbf{a}} = \Vert w\circ r\Vert_{\beta,\textbf{a}}.
\]
\end{prop}

\begin{proof}
Using \eqref{Equation:IsoparametricLaplacian}, we obtain  that
\begin{equation}\label{Eq:IsoparametricLaplacian}
\Delta_g^{i}(w\circ r) = \left( \mathcal{L}^{i} w \right)\circ r,\qquad i\in\mathbb{N}\cup\{0\},
\end{equation}
and therefore
\begin{align*}
\vert \nabla (\Delta_g^{i} w\circ r) \vert_g^2 &= \vert \nabla\left((\mathcal{L}^{i}w)\circ r\right)\vert^2_g\\
&= \left\langle \nabla\left((\mathcal{L}^{i}w)\circ r\right),\nabla\left((\mathcal{L}^{i}w)\circ r\right) \right\rangle_g\\
&= \left\vert \left((\mathcal{L}^{i}w)\right)'\circ r \right\vert^2 \vert \nabla r \vert^2\\
&= \left\vert \left(\mathcal{L}^{i}w\right)' \right\vert^2\circ r.
\end{align*}
By Lemma \ref{Lemma:IntegralFormula} this implies that 
\[
\int_M \vert \Delta_g^{i} (w\circ r) \vert^2 \; dV_g = \int_0^d \vert \mathcal{L}^{i}w\vert^2 \beta(t) \; dt
\quad \text{ and } \quad
\int_M \vert \nabla\Delta_g^{i} (w\circ r)\vert_g^2 \; dV_g = \int_0^d \left\vert \left( \mathcal{L}^{i}w \right)'\right\vert^2 \beta(t) \;dt,
\]
for every $i\in\mathbb{N}$, while for $i=0$ we get
\[
\int_M a_0 \vert w\circ r\vert^2\; dV_g = \int_M (\alpha_0\circ r) \vert w\circ r\vert^2\; dV_g = \int_0^d \alpha_0\vert w\vert^2 \beta(t)\ dt
\]

In this way, using \eqref{Equation:PowerLaplacian} and \eqref{Equation:InteriorProductOperator} we obtain that
\begin{align*}
&\Vert w\circ r\Vert_{g,\textbf{a}}^2 = \int_M (w\circ r) P_{\textbf{a}}(w\circ r) \; dV_g\\
&= \sum_{\substack{i\neq0\\ i\  even}}^m a_i \int_M \vert \Delta_g^{i/2} (w\circ r)\vert^2 \; dV_g
+ \sum_{\substack{i=0\\ i\  odd}}^m a_i\int_M \vert \nabla\Delta_g^{(i-1)/2}(w\circ r) \vert_g^2 \; dV_g + \int_M a_0 \vert w\circ r\vert^2\; dV_g\\
&=\sum_{\substack{i\neq0\\ i\  even}}^m a_i \int_0^d \vert \mathcal{L}^{i/2} w \vert^2\beta(t) \;dt
+ \sum_{\substack{i=0\\ i\  odd}}^m a_i\int_0^d \left( \left\vert \mathcal{L}^{(i-1)/2}w \right)'\right\vert^2 \beta(t) \;dt + \int_0^d \alpha_0\vert w\vert^2 \beta(t)\ dt\\
&= \Vert w\Vert_{\beta,\textbf{a}},
\end{align*}
as we wanted to prove.
\end{proof}

From this result, it follows that $\Vert\cdot\Vert_{\beta,\textbf{a}}$ is a well defined norm in $C^\infty[0,d]$ and we define the {\em weighted Sobolev space} $H_\beta^m(0,d)$ to be the closure of $C^\infty[0,d]$ under this norm.

We have the following direct consequence of the previous result.

\begin{theorem}
For any $\textbf{a}\in \mathcal{C}_+^\infty(M)^\Gamma\times(0,\infty)^{m}$, the linear isomorphism $\iota$ given in \eqref{Eq:IsomorphismSmoothFunctionSpace}, induces a well defined continuous isometric isomorphism 
\[
\iota : \left(  H_g^m(M)^\Gamma, \Vert \cdot \Vert_{g,\textbf{a}} \right) \rightarrow\left(  H_\beta^m(0,d), \Vert\cdot\Vert_{\beta,\textbf{a}} \right).
\]
\end{theorem}

\begin{proof}
Take $C^\infty(M)^\Gamma$ and $C^\infty([0,d])$ as dense subspaces of $H_g^m(M)^\Gamma$ and $H_\beta^m(0,d)$ under the norms $\Vert\cdot\Vert_{g,\textbf{a}}$ and $\Vert\cdot\Vert_{\beta,\textbf{a}}$, respectively. By Proposition \ref{Proposition:EqualityNorms}, given any $u\in C^\infty(M)^\Gamma$, $u=w\circ r$, we have that $\Vert u\Vert_{g,\textbf{a}}= \Vert w\Vert_{\beta,\textbf{a}}= \Vert \iota(u)\Vert_{\beta,\textbf{a}}$, and the map $\iota:C^\infty(M)^\Gamma\rightarrow C^\infty([0,d])$ is a linear and continuous isometric isomorphism. Thus, $\iota$ can be extended, in a unique way, to a linear and continuous isometric isomorphism defined on the whole Sobolev space $H_g^m(M)^\Gamma$, as we wanted to prove.
\end{proof}
\medskip

Next we see how the standard norm in $H^m(0,d)$ is related with the weighted norms $\Vert \cdot \Vert_{\beta,\textbf{a}}.$

\begin{lemma}\label{Lemma:EquivalenceWeightedNorms}
For each $\varepsilon>0$, there exist $\textbf{k}=(k_0,\ldots,k_m)\in \mathcal{C}_+^\infty(M)^\Gamma\times(0,\infty)^{m}$, with $k_0$ constant, and $A,B>0$, depending on $\varepsilon$, such that
\[
B \Vert w\Vert_{H^m(\varepsilon,d-\varepsilon)}\geq\Vert w\Vert_{\beta,\textbf{k}} \geq A \Vert w\Vert_{H^m(\varepsilon,d-\varepsilon)}
\]
\end{lemma}

\begin{proof}
As $\beta$ is continuous and positive in $(0,d)$, $\max_{[\varepsilon,d-\varepsilon]}\beta>0$. Then, it is readily seen that, for any $\textbf{a}=(\alpha_0\circ r, a_1,\ldots,a_m)\in \mathcal{C}_+^\infty(M)^\Gamma\times(0,\infty)^{m}$,  the inequality \[
\Vert w\Vert_{\beta,\textbf{a}}\leq B \Vert w\Vert_{H^m(\varepsilon,d-\varepsilon)}
\]
holds true, where $B$ is a suitable constant depending only on $\max_{i}a_i$, $\max_{[\varepsilon,d-\varepsilon]}\beta$ and  $\max_{[\varepsilon,d-\varepsilon]}\alpha_0\beta>0$. 
\newline

The proof of the second inequality is exactly the same as in \cite[Lemma 2.3]{ClappFernandezSaldana2021}.

\end{proof}

\begin{corollary}\label{Corollary:Regularity}
For any $u\in H_g^m(M)^\Gamma$ there exists $\tilde{u}\in C^{m-1}(M\setminus(M_-\cup M_+))^\Gamma$ such that
\[
u = \tilde{u},\quad \text{ a.e. in }M.
\]
\end{corollary}

\begin{proof}
The proof is virtually the same as in \cite[Proposition 3]{ClappFernandezSaldana2021}, but we include it for the sake of completeness. 

First observe that for any $\textbf{a}\in \mathcal{C}_+^\infty(M)\times(0,\infty)^{m}$, the operator $P_{\textbf{a}}$ is coercive and, therefore  the norm $\Vert\cdot\Vert_{g,\textbf{a}}$ is equivalent to the standard norm in $\Vert\cdot\Vert_{H^m_g(M)}$. 
Next, fix $\varepsilon>0$, take $\textbf{k}\in \mathcal{C}_+^\infty(M)^\Gamma\times(0,\infty)^{m}$ and $A,B>0$ as in Lemma \ref{Lemma:EquivalenceWeightedNorms} and let $\Omega_{\varepsilon, d-\varepsilon}:= r^{-1}(\varepsilon, d-\varepsilon)$. Since $\Vert\cdot\Vert_{g,\textbf{a}}$ is equivalent to the standard norm in $H_g^m(M)$, the map $\iota:( H_g^m(\Omega_{\varepsilon, d-\varepsilon})^\Gamma,\Vert\cdot\Vert_{H_g^m(M)})\rightarrow (H^m(\varepsilon, d-\varepsilon), \Vert \cdot\Vert_{H^m})$ is continuous. Since the Sobolev embedding $H^m(\varepsilon,d-\varepsilon)\hookrightarrow C^{m-1}(\varepsilon,d-\varepsilon)$ is also continuous, for any $u\in H_g^m(M)^\Gamma$ and  $w\in H_{\beta}^m(0,d)$ such that $u=w\circ r$, there exists $w_\varepsilon\in C^{m-1}[\varepsilon, d- \varepsilon]$ such that $w=w_\varepsilon$ a.e. in $[\varepsilon,d-\varepsilon]$. Applying Lemma \ref{Lemma:MeasureZero} it follows that $u = u_\varepsilon$ a.e. in $\Omega_{\varepsilon, d-\varepsilon}$, where $u_\varepsilon=w_\varepsilon \circ r\in C^{m-1}(\Omega_{\varepsilon, d-\varepsilon})$. As $M_-\cup M_+$ has measure zero in $M$, the function $u:M\setminus(M_-\cup M_+)\rightarrow\mathbb{R}$ given by $\tilde{u}(p):=u_\varepsilon(p)$ if $p\in\Omega_{\varepsilon, d-\varepsilon}$, is well defined, is of class $C^{m-1}$ on $M\setminus(M_-\cup M_+)$ and coincides a.e. with $u$ on $M$.
\end{proof}

For any $a,b\in(0,d)$, let $\Omega_{a,b}:= r^{-1}(a,b)$. We now show that the Dirichlet boundary problem
\begin{equation}\label{Problem:DirichletBoundaryBigDimension}
\begin{cases}
P_g u = \vert u\vert^{p-2}u, & \text{ in }\Omega_{a,b},\\
\nabla^ku=0, k=0,\ldots, 2m-1, & \text{ on }\partial\Omega_{a,b},
\end{cases}
\end{equation}
induces a one dimensional Dirichlet boundary problem. We need some preliminary lemmas.
\newline

\begin{lemma}
Let $u=w\circ r$ for a smooth function $w:[0, d]\to \mathbb{R}$. Then for $k\geq 1$,
\[
\nabla^k u=\sum_{j=0}^{k-1}(w^{(k-j)}\circ r)\ T^{k,j+1},
\]
where $w^{(i)}$ denotes the $i$-th derivative of $w$ over $\mathbb{R}$,  $T^{k,j+1}$ is a $k$ tensor, varying with $k$, which is a combination of tensor products of the tensors $\nabla r,\nabla^2r,\dots,\nabla^{j+1} r$. Moreover, for any $k$,
\[
T^{k,1}=\nabla r\otimes\cdots\otimes\nabla r, \qquad (k\text{ factors}).
\]
\end{lemma}

\begin{proof}
We will proceed by induction over $k$.
\medskip

\textbf{Case $k=1$.} Denote by $X$ an arbitrary vector field that is tangent to the level sets $M_t$ of the distance function $r$. By definition of gradient and by the chain rule, 
\begin{eqnarray*}
\langle \nabla u, X\rangle_g&=&X(u)\\
&=&(w'\circ r)X(r)=0.
\end{eqnarray*}
Analogously, computing in the direction of the normal to $M_t$ and using the fact that $|\nabla r|^2_g=1$, we get that 
\begin{eqnarray*}
\langle\nabla u, \nabla r\rangle_g&=&\nabla r(u)\\
&=&(w'\circ r)\nabla r(r)\\
&=&(w'\circ r) \langle\nabla r,\nabla r\rangle_g\\
&=&(w'\circ r) \vert \nabla r\vert_g\\
&=&(w'\circ r).
\end{eqnarray*}
 Therefore,
\[
\nabla u= (w'\circ r)\nabla r.
\]

\textbf{Case $k=2$.} Recall that the Levi-Civita connection induces a covariant derivative for higher-order tensors. Given a tensor $T$ of order $k$, the derivative $\nabla T$ is a tensor of order $(k+1)$ given by the formula:
\begin{equation*}
\begin{split}
\nabla T(X_1, \dots, X_k,X_{k+1})=&X_{k+1}(T(X_1, \dots, X_k))-T(\nabla_{X_{k+1}}X_1, \dots, X_k)-...\\
&-T(X_1, \dots, \nabla_{X_{k+1}}X_k).
\end{split}
\end{equation*}
We then compute for any vector fields $X_1, X_2$  on $M$:
\begin{eqnarray*}
\nabla^2 u(X_1, X_2)&=& \nabla \left( (w'\circ r)\nabla r\right)(X_1, X_2)\\
&=&X_2 \left((w'\circ r)\nabla r(X_1) \right)-(w'\circ r)\nabla r\left(\nabla_{X_2}X_1\right)\\
&=& (w'\circ r)\left[X_2(\nabla r(X_1))-\nabla r\left(\nabla_{X_2} X_1\right) \right]+X_2(w'\circ r)\nabla r(X_1)\\
&=&(w'\circ r)\nabla^2 r(X_1, X_2)+X_2(w'\circ r)\nabla r(X_1)\\
& =&(w'\circ r)\nabla^2 r(X_1, X_2)+(w''\circ r)\nabla r(X_1)\nabla r(X_2).
\end{eqnarray*}
Therefore:
\[
\nabla^2u=(w'\circ r)\nabla^2 r+(w''\circ r)\nabla r\otimes \nabla r.
\]

\textbf{Case $k\geq 3$.} Now, suppose that
\[
\nabla^{k-1} u=\sum_{j=0}^{k-1}(w^{(k-1-j)}\circ r)T^{k-1,j+1},
\]
where the tensors $T^{k-1,j+1}r$ satisfy the conditions in the Lemma. Take any $X_1, \dots, X_k$ vector fields on $M$ and compute:
\begin{equation}
\label{eq:induction}
\begin{aligned}
\nabla^ku(X_1,\dots,X_k)& = X_k(\nabla^{k-1}u\left(X_1,\dots,X_{k-1}\right))-\nabla^{k-1}u(\nabla_{X_k}X_1, \dots, X_{k-1})\\
&\quad  -\cdots-\nabla^{k-1}u(X_1, \dots, \nabla_{X_k}X_{k-1})
\end{aligned}
\end{equation}

We substitute the expression for $\nabla^{k-1} u$. For the sake of clarity, let us analyze the first summand:
\begin{align*}
 X_k(\nabla^{k-1}u\left(X_1,\dots,X_{k-1}\right))  &= X_k\left(\sum_{j=0}^{k-1}(w^{(k-1-j)}\circ r)T^{k-1,j+1}(X_1, \dots, X_{k-1})\right) \\
 & = \sum_{j=0}^{k-1} X_k(w^{(k-1-j)}\circ r)T^{k-1,j+1}(X_1, \dots, X_{k-1})  \\
 & \quad +  \sum_{j=0}^{k-1} (w^{(k-1-j)}\circ r) X_k(T^{k-1,j+1}(X_1, \dots, X_{k-1})).
 \end{align*}
 Now,  
\[
X_k(w^{(k-1-j)}\circ r)=(w^{(k-j)}\circ r)X_k(r).
\]
Note that $j=0$ gives the only term with the factor $(w^{(k)}\circ r)$ in the expression for $\nabla^ku$. Observe also that $X_k(T^{k-1,j+1}(X_1, \dots, X_{k-1}))$ is one of the terms in the definition of $\nabla T^{k-1,j+1}(X_1, \dots, X_k)$; the others will be obtained from the remaining terms in \eqref{eq:induction}. If $T^{k-1,j+1}$ is a combination of tensor products of $\nabla r,\nabla^2 r,\dots,\nabla^{j+1}$, the same happens with $\nabla T^{k-1,j+1}$. After a long but straightforward calculation, we have
 \begin{align*}
 \nabla^ku(X_1,\dots,X_k) & =(w^{(k)}\circ r)X_k(r) T^{k-1,1}(X_1, \dots, X_{k-1}) \\
 & \quad + \sum_{j=0}^{k-2}(w^{(k-1-j)}\circ r)T^{k,j+1}(X_1,\dots,X_k),
 \end{align*}
 for some tensors $T^{k,j+1}$. By the inductive hypothesis,
 \begin{align*}
 X_k(r) T^{k-1,1}(X_1, \dots, X_{k-1})& =X_k(r)(\nabla r\otimes\cdots\otimes\nabla r )(X_1, \dots, X_{k-1}) \quad \left( (k-1)\text{ factors}\right) \\
 & = (\nabla r\otimes\cdots\otimes\nabla r )(X_1, \dots, X_k) \quad (k \text{ factors}),
 \end{align*}
 which proves the lemma.
\end{proof}
\medskip

\begin{prop}\label{Proposition:ZeroDerivative}
Let $u=w\circ r$, for a smooth function $w\colon[0, d]\to \mathbb{R}$. Let $x\in M$, and a fixed integer $k\geq 1$ such that $\nabla^l u(x)=0$ for all $1\leq l\leq k$, then $(w^{(l)}\circ r)(x)=0$, for all $1\leq l \leq k$.
\end{prop}
\begin{proof}
We will proceed by induction over $k$.  For $k=1$, we have:
\[
\nabla u=(w'\circ r)\nabla r.
\]
Evaluating at $x$ and $\nabla r$,
\[
0=\nabla u(x)(\nabla r)=(w'\circ r)(x)\nabla r(\nabla r)(x)=(w'\circ r)(x).
\]

Now take $k>1$ and suppose that $\nabla^l u(x)=0$, $(w^{(l)}\circ r)(x)=0$ for all $1\leq l< k$ and $\nabla^k u(x)=0$. Using this and the previous lemma,
\[
\nabla^k u=\sum_{j=0}^{k-1}(w^{(k-j)}\circ r)T^{k,j+1}=(w^{(k)}\circ r)T^{k,1}=(w^{(k)}\circ r)(\nabla r\otimes\cdots\otimes\nabla r )
\]
Evaluating at $x$ and $(\nabla r, \dots, \nabla r)$,
\[
0=\nabla^k u(x)(\nabla r, \dots, \nabla r)=(w^{k}\circ r)(x).
\]
The result follows.
\end{proof}
\medskip

\begin{corollary}\label{Corollary:EquivalenceDirichletBoundaryProblem}
Let $(a_0,a_1,\ldots,a_m)\in \mathcal{C}_+^\infty(M)^\Gamma\times(0,\infty)^m$ such that $P_g = \sum_{i=0}^m a_i(-\Delta_g)^{i}$ and define the operator
\[
\widehat{\mathcal{L}} = \alpha_0 + \sum_{i=1}^m a_i(- \mathcal{L})^{i},
\]
where $a_0=\alpha_0\circ r.$
If $u=w\circ r\in C^{2m}(\Omega_{a,b})$ is a solution to \eqref{Problem:DirichletBoundaryBigDimension}, then $w$ is a solution to the problem
\begin{equation}\label{Problem:DirichletBoundaryOneDimension}
\begin{cases}
\widehat{\mathcal{L}}w = \vert w\vert^{p-2}w, & \text{ in }(a,b),\\
w^{(k)}(a)=w^{(k)}(b)=0, &  k=0,\ldots, 2m-1. 
\end{cases}
\end{equation}
\end{corollary}

\begin{proof}
Let $u=w\circ r$ be a smooth $\Gamma$-invariant solution to \eqref{Problem:DirichletBoundaryBigDimension}. From \eqref{Eq:IsoparametricLaplacian} it follows that
\[
P_g u = \sum_{i=0}^m a_i (-\Delta_g)^{i}u = \sum_{i=0}^{m}a_i\left( (-\mathcal{L})^{i}w \right)\circ r = \widehat{\mathcal{L}}(w)\circ r
\]
and $w$ satisfies $\widehat{\mathcal{L}}w=\vert w\vert^{p-2}w$ in $(a,b)$. Moreover, as $\nabla^ku(x)=0$ for every $0\leq k\leq m-1$ and every $x\in\partial\Omega_{a,b}$, by Proposition \ref{Proposition:ZeroDerivative},
\[
0 = w^{(k)}\circ r(x) =
\begin{cases}
w^{(k)}(a), & \text{if }x\in r^{-1}(a)\\
w^{(k)}(b), & \text{if }x\in r^{-1}(b),
\end{cases}
\]
and $w$ is a (strong) solution to \eqref{Problem:DirichletBoundaryOneDimension}.
\end{proof}

%%%%%%%%%%%%%%%%%%%%%%%%%%%%%%%%%%%%%
%%%%%%%%%%%%%%%%%%%%%%%%%%%%%%%%%%%%%%%%%%%%%%%%%%%%%%%%%%%%%%%%%%%%%%%%%%%

\section{Segregation and optimal partitions}\label{Section:Segregation}

In this section, we will suppose that $(M,g)$ is such that the operator $P_g$ can be written in the form \eqref{Eq:OperatorSumLaplacians} and that $\Gamma$ satisfies (\hyperref[Gamma:Cohomogeneity]{$\Gamma1$}) and (\hyperref[Gamma:DimensionOrbits]{$\Gamma2$}), so that the results of the previous section hold true. Remember that this is possible, for example, if $(M,g)$ is an Einstein manifold with positive scalar curvature or when $R_g>0$ in case $m=1$. See Section \ref{Section:Examples} for concrete examples.
\newline

Recall that for a compact Lie group $\Gamma$, a {\em principal $\Gamma$-bundle} is a fiber bundle $\Gamma\to P\to B$, whose structure group is $\Gamma$, together with a $\Gamma$-action on $\Gamma$ itself by left translations, and a free right $\Gamma$-action on $P$, whose orbits are the fibers of the bundle. Let $F$ be another smooth manifold that admits a left action by $\Gamma$. The orbit space of the diagonal action on $P\times F$ is a smooth manifold denoted by $P\times_{\Gamma} F$, given as the total space of the fiber bundle
    \[
    F\to P\times_{\Gamma} F\to B.
    \]
    In the literature, the latter is known as the {\em associated bundle} to the principal bundle $\Gamma\to P\to B$, and $P\times_{\Gamma} F$ is called the {\em twisted space}. See \cite[Section 3.1]{AlexBettiol} for definitions and a detailed explanation.
\newline

Let $\Omega$ be a $\Gamma$-invariant open subset of $M$ with smooth boundary and recall the definitions of the energy functional $J_\Omega$ and the Hilbert manifold $\mathcal{M}_\Omega^\Gamma$ given in Section \ref{Sec:Variational Setting}. By Proposition \ref{Proposition:ExistenceLeastEnergySolution}, problem \eqref{eq:dirichlet} admits a least energy $\Gamma$-invariant solution. So the quantity $c_\o^\Gamma$ defined in the introduction is attained.

Theorem \ref{Theorem:OptimalPartitionSymmetry} will follow from the next segregation result.

\begin{theorem}\label{Theorem:OptimalPartition} Suppose $\Gamma$ satisfies conditions \emph{(\hyperref[Gamma:Cohomogeneity]{$\Gamma1$})}, \emph{(\hyperref[Gamma:DimensionOrbits]{$\Gamma2$})} and \emph{(\hyperref[Gamma:MetricDecomposition]{$\Gamma3$})}, and that $P_g$ can be written as a sum of Laplacians of the form \eqref{Eq:OperatorSumLaplacians}. 
	For $i=1,\ldots,\ell$, fix $\nu_i=1$  and for each $i\neq j$, $k\in\n$, let $\eta_{ij,k}<0$ be such that $\eta_{ij,k}=\eta_{ji,k}$ and $\eta_{ij,k}\to -\infty$ as $k\to\infty$. Let $(u_{k,1},\ldots,u_{k,\ell})$ be a least energy fully nontrivial solution to the system \eqref{Eq:Q Systems} with $\eta_{ij}=\eta_{ij,k}$. Then, there exists $u_{\infty,1},\ldots u_{\infty,\ell}\in H_g^m(M)^\Gamma$ such that, up to a subsequence, 
\begin{itemize}
	\item[$(a)$]$u_{k,i}\to u_{\infty,i}$ strongly in $H^m_g(M)$,  $u_{\infty,i}\in\cC^{m-1}(M)$, $u_{\infty,i}\neq 0$.  Let 
	\begin{align*}
		\Omega_i:=\operatorname{int}\overline{\{x\in\rn:u_{\infty,i}(x)\neq 0\}}\qquad \text{ for \ }i=1,\ldots,\ell.
	\end{align*}
	Then  $u_{\infty,i}\in H_{0,g}^m(\Omega)^\Gamma$ is a least  energy solution of \eqref{eq:dirichlet} in $\Omega_i$ for each $i=1,\ldots,\ell$.
	\item[$(b)$]$\{\Omega_1,\ldots,\Omega_\ell\}\in\cP_\ell^\Gamma$ is a solution to the $\Gamma$-invariant $\ell$--optimal partition problem \eqref{Problem:PartitionProblem}  satisfying the following properties:
	 \begin{enumerate}
    \item $\Omega_i$ is smooth and connected for every $i=1,\ldots, \ell$, $\overline{\Omega}_i\cap\overline{\Omega}_{i+1}\neq\emptyset$, $\Omega_i\cap\Omega_j=\emptyset$ if $\vert i-j \vert\geq 2$ and $\overline{\Omega_1\cup\ldots\cup\Omega_\ell} = M$;
    \item 
    \[
    \Omega_1\approx G\times_{K-}D_{-},\quad \Omega_\ell\approx G\times_{K+}D_{+},\quad \partial \Omega_1\approx\partial\Omega_\ell\approx \Gamma/K;
    \]
    \item For each $i\neq 1,\ell$,
    \[
    \Omega_i\approx \Gamma/K\times(0,1),\quad\overline{\Omega}_i\cap\overline{\Omega}_{i+1}\approx\Gamma/K,\quad \text{and}\quad \partial\Omega_i\approx \Gamma/K \sqcup \Gamma/K,
    \]
    where $G\times_{K\pm}D_{\pm}$ denote disk bundles at the singular orbits. \end{enumerate}
	\end{itemize} 
\end{theorem}

To prove this theorem we will need the following lemma, which is a version of the unique continuation principle that is suitable to our situation.

\begin{lemma}\label{Lemma:UniqueContinuation}
Let $a,b\in(0,d)$ and let $u\in C^{2m}(\Omega_{a,b})$ be a $\Gamma$-invariant solution to the Dirichlet boundary problem \eqref{Problem:DirichletBoundaryBigDimension} in $\Omega_{a,b}:=r^{-1}(a,b)$. If $u=0$ in any subdomain of the form $\Omega_{c,d}:=r^{-1}(c,d)$, $c,d\in[a,b]$, then $u=0$ in the whole interval $[a,b]$.
\end{lemma}

\begin{proof}
As $u$ is a strong $\Gamma$-invariant solution to \eqref{Problem:DirichletBoundaryBigDimension}, $u= w\circ r$ for some $w:[0,d]\rightarrow\mathbb{R}$ and Corollary \ref{Corollary:EquivalenceDirichletBoundaryProblem}, $w\in C^{2m}[a,b]$ is a strong solution to \ref{Problem:DirichletBoundaryOneDimension}. Since $u=0$ in $\Omega_{c,d}$ and $r\neq 0$ in $M\smallsetminus M_{-}=r^{-1}(0,d]$, then necessarily $w(t)=0$ for any $t\in[c,d]$. 

Now, if there is no $t_1\in [a,c]\cup[d,b]$ such that $w(t)\neq0$, then $w\equiv0$ in $[a,b]$ and there is nothing to prove. If this is not the case, there must be $t_0\in(a,c)\cup(d,b)$ such that $w(t_0)\neq 0$. Without loss of generality, suppose that $t_0\in(a,c)$; therefore, there must exist $a<t_1<t_2\leq c$ such that $w\neq 0$ in $(t_1,t_2)$ and $w=0$ in $[t_2,c]$. As $w$ is of class $C^{2m}$, all its derivatives of lower order are continuous and it follows that $w^{(k)}(t_2)=0$ for every $k=0,1,\ldots, 2m-1$. By existence and uniqueness of the initial value problem
\[
\widehat{\mathcal{L}}w = \vert w\vert^{p-2}w,\qquad w^{(k)}(t_2)=0, k=0,1,\ldots,2m-1
\]
$w$ vanishes identically in a small neighborhood of $t_2$, contradicting that $w\neq0$ in $(t_1,t_2)$. Something similar holds true if $w(t_0)\neq 0$ for some $t_0\in [d,b)$. Hence $w=0$ in $[a,b]$, as we wanted to show.
\end{proof}

The following topological lemma will be useful in what follows.

\begin{lemma}\label{Lemma:PreimageConnectedSets}
Let $X, Y$ be two topological spaces and  $r:X\rightarrow Y$ be a quotient map. If $r^{-1}(y)$ is connected for every $y\in Y$, then $r^{-1}(B)$ is connected for every connected subset $B\subset Y$.
\end{lemma}

\begin{proof}
Let $B\subset Y$ be connected and consider $A:=r^{-1}(B)$. Let $f:A\rightarrow\mathbb{Z}_2:=\{-1,1\}$ be any continuous function. Since $r^{-1}(y)$ is connected for every $y\in B$, then for any $x_1,x_2\in X$, $r(x_1)=r(x_2)$ implies that $f(x_1)=f(x_2)$, for $f$ maps connected sets into connected sets. Therefore $f$ induces a continuous function $\hat{f}:B\rightarrow\mathbb{Z}_2$ such that $\hat{f}\circ r = f$. As $B$ is connected, $\hat{f}$ is constant and so is $f$. As $f:A\rightarrow\mathbb{Z}_2$ was an arbitrary continuous function, it follows that $A$ must be connected. \end{proof}

The following result allow us to describe the nodal domains in terms of the orbit structure. 
\newline

\begin{prop}
\label{prop: solOPchone}
Given a solution $\{\Theta_1, \dots, \Theta_\ell\}\in \mathcal{P}_\ell^\Gamma$ to the optimal $\Gamma$-invariant $\ell$-partition problem \eqref{Problem:PartitionProblem},  there exist points $a_1, \dots, a_{\ell-1}\in (0, d)$ such that:
\[
(0, d)\setminus \bigcup_{i=1}^{\ell}r(\Theta_i)=\{a_1, \dots, a_{\ell-1}\}.
\]
and, up to a relabeling,
\begin{eqnarray*}
\Omega_1:=\Theta_1\cup M_-&=&r^{-1}[0, a_1)\approx G\times_{K-}D_{-} \\
\Omega_i:=\Theta_i&=&r^{-1}(a_{i-1}, a_{i})\approx M_{d/2}\times(0,1) \quad \quad \mbox{if $i=2, \dots, \ell-1$}\\
\Omega_\ell:=\Theta_\ell\cup M_+&=&r^{-1}(a_{\ell-1}, d]\approx G\times_{K+}D_{+}
\end{eqnarray*}
Moreover, the sets $\Omega_1, \dots, \Omega_\ell$ satisfy properties (b.1) to (b.3) of Theorem \eqref{Theorem:OptimalPartition} and $\{\Omega_1,\ldots,\Omega_\ell\}$ is also a solution to the $\Gamma$-invariant $\ell$-optimal partition problem \eqref{Problem:PartitionProblem}. 
\end{prop}

\begin{proof}
First notice that every connected $\Gamma$-invariant open set must be of the form $r^{-1}(t,s)$ for some $t,s\in[0,d]$. Now take three points $t_1, t_2, t_3\in(0, d)$ and define the sets  $V_1=\pi^{-1}(t_1, t_2)$, $V_2=\pi^{-1}(t_2, t_3)$ and $V=\pi^{-1}(t_1, t_3)$. Note that these sets are $\Gamma$-invariant by construction, their boundaries are smooth because $r^{-1}(t_i)$ is a principal orbit, and since every orbit is connected and $r$ is a quotient map, then $V_1,V_2$ and $V$ are also connected by Lemma \ref{Lemma:PreimageConnectedSets}. Hence, by Proposition \ref{Proposition:ExistenceLeastEnergySolution}, the least energy solution to the problem  \eqref{eq:dirichlet} is attained in each domain and
\[
c_V^\Gamma \leq \min\{c_{V_1}^\Gamma, c_{V_2}^\Gamma\},
\]
where this inequality holds true because $V_i\subset V$ and every nontrivial function in $H_{0,g}^m(V_i)^\Gamma$ can be extended by zero to a nontrivial function in $H_{0,g}^m(V_i)^\Gamma$. We next prove that the inequality is strict. Suppose, to get a contradiction, and without loss of generality, that $c_V^\Gamma = c_{V_1}^\Gamma$ and let $u\in H_{0,g}(V_1)^\Gamma$ be a least energy solution to the Dirichlet boundary problem \eqref{eq:dirichlet} in $V_1$. Therefore the function $\hat{u}\in H_{0,g}^m(V)^\Gamma$ given by $\hat{u}=u$ in $V_1$ and $\hat{u}=0$ in $V\smallsetminus V_1$ is a least energy solution to \eqref{eq:dirichlet} in $V$ and by interior regularity \cite{UhlenbeckViaclovsky2000}, this function has a $C^{2m}$ class representative. Proposition \ref{Lemma:UniqueContinuation} yields that $\hat{u}$ must vanish in $V$, which is a contradiction and the strict inequality follows.

Hence, if $\{\Theta_1,\ldots,\Theta_\ell\}$ is a $\Gamma$-invariant solution to the $\ell$-partition problem \eqref{Problem:PartitionProblem}, then $(0,d)\setminus \cup^\ell_{i=1} r(\Theta_i)$ consists exactly of $\ell-1$ points, say $a_1,\dots, a_{\ell-1}$.

Now define $\Omega_i$ as in the statement. As $r^{-1}(t)$ is either a connected principal orbit or a connected singular orbit for each $t\in[0,d]$, by Lemma \ref{Lemma:PreimageConnectedSets} these sets $\Omega_i$ are connected and $\partial\Omega_i$ consists in one or two disjoint principal orbits, from which it follows that these sets are smooth. As $r(\overline{\Omega}_i)=[a_{i-1}, a_{i}]$, for $i=1,\ldots \ell$ (where $a_0:= 0$ and $a_\ell = d$), then $\overline{\Omega_1\cup\ldots\cup\Omega_\ell} = M$. By definition, $\Omega_i\cap\Omega_j=\emptyset$ if $\vert i - j \vert\geq 2$, $\overline{\Omega}_{i}\cap\overline{\Omega}_{i+1}= r^{-1}(i)\approx M_{d/2}\approx \Gamma/K$, and
\[
\Omega_i = \begin{cases}
r^{-1}(a_i,a_{i+1})\approx M_{d/2}\times(a_{i-1},a_{i})\approx \Gamma/K\times(0,1), & i=2,\ldots,\ell-1,\\
 r^{-1}[0,a_1)\approx G\times_{K-}D_{-}, & i=1,\\
 r_1^{-1}(a_{\ell-1},d]\approx G\times_{K+}D_{+}, & i=\ell.
\end{cases}
\]
The equivariant form of the sets $\Omega_1$ and $\Omega_{\ell}$ follows from the Tubular Neighborhood Theorem \cite[Theorem 3.57]{AlexBettiol}. In fact, they are the associated bundles to the principal $K$-bundle $K\to\Gamma\to\Gamma/K$.
\newline

Finally as $\Theta_i\subset\Omega_i$, then $c_{\Omega_i}^\Gamma\leq c_{\Theta_i}^\Gamma$ and $\{\Omega_1,\ldots,\Omega_\ell\}\in\mathcal{P}_\ell^\Gamma$ is also a solution to problem \eqref{Problem:PartitionProblem}.
\end{proof}
\medskip

\begin{proof}[Proof of Theorem \ref{Theorem:OptimalPartition}] 
With minor modifications, the proof is the same as in \cite[Theorem 1.2]{ClappFernandezSaldana2021}. We sketch it for the reader's convenience.

Fix $\nu_i=1$ in \eqref{Eq:Q Systems} for each $i=1,\ldots,\ell$, and let $(\eta_{ij,k})_{k\in\mathbb{N}}$ be a sequence of negative numbers such that $\eta_{ij,k}=\eta_{ji,k}$ and $\eta_{ij,k}\to-\infty$ as $k\to\infty$. To highlight the role of $\eta_{ij,k}$, we write $\mathcal{J}_k$ and $\mathcal{N}_k$ for the functional and the set associated to the system \eqref{Eq:Q Systems}, introduced in Section~\ref{sec:system}, with $\eta_{ij}$ replaced by $\eta_{ij,k}$. By Theorem \ref{Th:MainQSystems}, for each $k\in\mathbb{N}$ we can take $\overline{u}_k=(u_{k,1},\ldots,u_{k,\ell})\in\cN_k$ such that
	$$c_k^\Gamma:= \inf_{\mathcal{N}_k} \mathcal{J}_k =\mathcal{J}_k(\overline{u}_k)=\frac{m}{N}\sum_{i=1}^\ell\|u_{k,i}\|_{P_g}^2.$$
	Let
	\begin{align*}
		\mathcal{N}_0:=\{(v_1,\ldots,v_\ell)\in\mathcal{H}:\,&v_i\neq 0,\;\|v_i\|_{P_g}^2=\int_{M}|v_i|^{{2^*_m}}, \text{ and }v_iv_j=0\text{ a.e. in }M \text{ if }i\neq j\}.
	\end{align*}
	Then, $\mathcal{N}_0\subset\mathcal{N}_k$ for all $k\in\mathbb{N}$ and, therefore, 
	\begin{equation}\label{Eq:InequalityC_0}
	0<c_k^\Gamma\leq c_0^\Gamma:=\inf\left\{\frac{m}{N}\sum_{i=1}^\ell\|v_i\|_{P_g}^2:(v_1,\ldots,v_\ell)\in\mathcal{N}_0\right\}<\infty.
	\end{equation}
	
	We claim that 
	\begin{equation}\label{Claim:InfimumOptimalPartition}
	    c_0^\Gamma \leq \inf_{\{ \Phi_1,\ldots,\Phi_\ell \}\in\mathcal{P}_\ell^\Gamma} \sum_{i=1}^\ell c_{\Phi_i}^\Gamma
	\end{equation}
	
	Indeed, if $\{\Phi_1,\ldots,\Phi_\ell\}\in\mathcal{P}_\ell^\Gamma$ and $v_i\in\mathcal{M}_{\Phi_i}^\Gamma\subset H_{0,g}^m(\Phi_i)^\Gamma$, then extending this function by zero outside $\Phi_i$, we get that $v_i\in H^m_g(M)^\Gamma$ and $v_i,v_j = 0$ a.e. in $M$, for $\Phi_i\cap \Phi_j=\emptyset$. Therefore $\overline{v}:=(v_1,\ldots,v_\ell)\in\mathcal{N}_0\subset\mathcal{N}_1$ and
	\[
	c_0^\Gamma\leq \frac{m}{N}\Vert v_i\Vert^2_{P_g} = \mathcal{J}_1(\overline{v}) = \sum_{i=1}^\ell J_{\Phi}(v_i).
	\]
	As $v_i\in\mathcal{M}_{\Phi_i}^\Gamma$ was arbitrary, it follows that
	\[
	c_0^\Gamma \leq \sum_{i=1}^\ell c_{\Phi_i}^\Gamma,
	\]
	and as $\{\Phi_1,\ldots,\Phi_\ell\}\in\mathcal{P}_\ell^\Gamma$ was arbitrary, inequality  \eqref{Claim:InfimumOptimalPartition} follows.

	From \eqref{Eq:InequalityC_0}, it follows that the sequence $(\overline{u}_k)$ is bounded in $\mathcal{H}$. So, using Lemma~\ref{Lemma:Sobolev}, after passing to a subsequence, we get that $u_{k,i} \rightharpoonup u_{\infty,i}$ weakly in $H_{g}^{m}(M)^\Gamma$, $u_{k,i} \to u_{\infty,i}$ strongly in $L_g^{{2^*_m}}(M)$, and $u_{k,i} \to u_{\infty,i}$ a.e. in $M$ for each $i=1,\ldots,\ell$.  Moreover, as $\partial_i\mathcal{J}_k(\overline{u}_k)[u_{k,i}]=0$, we have for each $j\neq i$,
	\begin{align*}
		0\leq\int_{M}\beta_{ij}|u_{k,j}|^{\alpha_{ij}}|u_{k,i}|^{\beta_{ij}}\leq \frac{1}{-\eta_{ij,k}}\int_{M}|u_{k,i}|^{{2^*_m}}\leq \frac{C}{-\eta_{ij,k}}.
	\end{align*}
	Then, Fatou's lemma yields 
	$$0 \leq \int_{M}|u_{\infty,j}|^{\alpha_{ij}}|u_{\infty,i}|^{\beta_{ij}} \leq \liminf_{k \to \infty} \int_{M}|u_{k,j}|^{\alpha_{ij}}|u_{k,i}|^{\beta_{ij}} = 0.$$
	Hence, $u_{\infty,j} u_{\infty,i} = 0$ a.e. in $M$. By Lemma \ref{lem:away_froM_{d/2}},
	$$0<d_0 \leq \|u_{k,i}\|_{P_g}^2 \leq\int_{M} |u_{k,i}|^{{2^*_m}}\qquad\text{for all \ }k\in\mathbb{N},\;i=1,\ldots,\ell,$$
	and, as $u_{k,i} \to u_{\infty,i}$ strongly in $L^{{2^*_m}}(M)$ and $u_{k,i} \rightharpoonup u_{\infty,i}$ weakly in $H_{g}^m(M)$, we get
	\begin{equation} \label{eq:comparison2}
		0<\|u_{\infty,i}\|_{P_{g}}^2 \leq \int_{M}|u_{\infty,i}|^{{2^*_m}}\qquad\text{for every \ }i=1,\ldots,\ell.
	\end{equation}
	Since $u_{\infty,i}\neq 0$, there is a unique $t_i\in(0,\infty)$ such that $\|t_iu_{\infty,i}\|_{P_g}^2 = \int_{M}|t_iu_{\infty,i}|^{{2^*_m}}$. So $(t_1u_{\infty,1},\ldots,t_\ell u_{\infty,\ell})\in \mathcal{N}_0$. The inequality \eqref{eq:comparison2} implies that $t_i\in (0,1]$. Therefore,
	\begin{align*}
		c_0^\Gamma &\leq \frac{m}{N}\sum_{i=1}^\ell\|t_iu_{\infty,i}\|_{P_g}^2 \leq \frac{m}{N}\sum_{i=1}^\ell\|u_{\infty,i}\|_{P_g}^2\leq \frac{m}{N}\liminf_{k\to\infty}\sum_{i=1}^\ell\|u_{k,i}\|_{P_g}^2=\liminf_{k\to\infty} c_k^\Gamma \leq c_0^\Gamma.
	\end{align*}
	It follows that $u_{k,i} \to u_{\infty,i}$ strongly in $H_{g}^m(M)^\Gamma$ and $t_i=1$, yielding 
	\begin{equation}\label{eq:limit}
		\|u_{\infty,i}\|_{P_g}^2 = \int_{M}|u_{\infty,i}|^{{2^*_m}},\qquad\text{and}\qquad\frac{m}{N}\sum_{i=1}^\ell\|u_{\infty,i}\|_{P_g}^2 =c_0^\Gamma.
	\end{equation}
	
	By Corollary \ref{Corollary:Regularity}, we can take $u_{\infty,i}\in C^{m-1}(M\smallsetminus (M_+\cup M_-))$ so that $u_{\infty,i}u_{\infty,j}=0$ in $M\smallsetminus(M_-\cup M_+)$, $i\neq j$. It follows from continuity that the set
	\[
	\Theta_i:=\{ x\in M\smallsetminus(M_-\cup M_+)\;:\; u_{\infty,i}\neq 0 \}, \ i=1,\ldots,\ell
	\]
	is nonempty, open, $\Gamma$-invariant and $\Theta_i\cap\Theta_j=\emptyset$ if $i\neq j$. 
	
	Set 
	\[
	\Omega_i = \text{int}(\overline{\Theta}_i), \ i=1,\ldots,\ell.
	\]
	These sets are also nonempty, $\Gamma$-invariant and open, and satisfy that $\Omega_i\cap\Omega_j = \emptyset$ if $i\neq j$ and $u_{\infty,i}=0$ in $M\smallsetminus{\Omega_i}$. Hence $\{\Omega_1,\ldots,\Omega_\ell\}\in\mathcal{P}_\ell^\Gamma$. Since each connected component of $\partial\Omega_i$ is of the form $r^{-1}(t)\approx M_{d/2}$ for some $t\in (0,d)$, it is smooth and
	\[
	H_{0,g}^m(\Omega_i)=\{u\in H_g^m(M) \; : \; u=0 \text{ in } M\smallsetminus\Omega_i\}.
	\] (Cf.  \cite[Lemma A.1]{ClappFernandezSaldana2021} and \cite[Theorem 1.4.2.2]{GrisvardBook}). Hence, as $u_{\infty,i}=0$ in $M\smallsetminus\Omega_i$, $u_{\infty,i}\neq 0$ in $\Omega_i$ and satisfies \eqref{eq:limit}, it follows that $u_{\infty,i}\in \mathcal{M}_{\Omega_i}^\Gamma\subset H_{0,g}^m(\Omega)^\Gamma$. As $c_{\Omega_i}\leq J_{\Omega_i}(u_{\infty,i})$, using the claim \eqref{Claim:InfimumOptimalPartition} we obtain that
	\begin{equation}\label{Eq:FundamentalInequality}
	\begin{split}
	\inf_{\{ \Phi_1,\ldots,\Phi_\ell \}\in\mathcal{P}_\ell^\Gamma} &\sum_{i=1}^\ell c_{\Phi_i}^\Gamma 
	\leq \	\sum_{i=1}^\ell c_{\Omega_i}^\Gamma
	\leq \sum_{i=1}^\ell J_{\Omega_i}(u_{\infty,i})\\ &= \frac{m}{N}\sum_{i=1}^\ell\|u_{\infty,i}\|^2 = c_0^\Gamma \leq \inf_{(\Phi_1,\ldots,\Phi_\ell)\in\mathcal{P}_\ell^\Gamma}\;\sum_{i=1}^\ell c_{\Phi_i}^\Gamma.
	\end{split}
	\end{equation}
	Also, from this inequality we obtain that $J_{\Omega_i}(u_{\infty,i})=c_{\Omega_i}^\Gamma$ for every $i=1,\ldots,\ell$, for otherwise, we would get that second inequality in \eqref{Eq:FundamentalInequality} is strict, yielding a contradiction. Hence $u_{\infty,i}$ is a (weak) solution to the Dirichlet boundary problem \eqref{eq:dirichlet} and $\{\Omega_1,\ldots,\Omega_\ell\}$ is a solution to the $\Gamma$-invariant $\ell$-partition problem \eqref{Problem:PartitionProblem}. Proposition \ref{prop: solOPchone} yields that, actually, $\Omega_1,\ldots,\Omega_\ell$ satisfy properties (b.1) to (b.3) in Theorem \ref{Theorem:OptimalPartition}.
\end{proof}	
	
	\begin{remark}\label{Remark:OptimalPartitionPowerP}
	Changing the exponent $p=2_m^\ast$ by any $2\leq p\leq 2_m^\ast$, the arguments in this section yield a solution $\{\Omega_1,\ldots,\Omega_\ell\}$ to the $\Gamma$-invariant $\ell$-partition problem associated to the more general Dirichlet boundary problem \eqref{Problem:DirichletBoundary}, where the sets $\Omega_i$ satisfy properties (b.1) to (b.3) in Theorem \ref{Theorem:OptimalPartition}.\qed
	\end{remark}
	
For the case $m=1$, that is, when $P_g=-\Delta_g + R_g$ is just the conformal Laplacian, we have the following result from which Corollary \ref{Corollary:YamabeProblem} follows immediately.

\begin{corollary}
Let $(M,g)$ be a closed Riemannian manifold of dimension $N\geq3$ and let $\Gamma$ be a closed subgroup of $\text{Isom}(M,g)$ satisfying \emph{\hyperref[Gamma:Cohomogeneity]{$(\Gamma1)$}} to \emph{\hyperref[Gamma:MetricDecomposition]{$(\Gamma3)$}}. If the scalar curvature $R_g$ is positive, and if $\{\Omega_1, \dots, \Omega_\ell\}\in\mathcal{P}_\ell^\Gamma$ is  the solution to the optimal $\Gamma$-invariant  $\ell$-partition problem given in Theorem \ref{Theorem:OptimalPartition}, then the function
\[
u_\ell := \sum_{i=1}^\ell (-1)^{i}u_{\infty,i}
\]
is a $\Gamma$-invariant sign-changing solution to the Yamabe problem
\[
-\Delta_g u + \frac{N-2}{4(N-1)}R_g u = \vert u\vert^{2_1^\ast-2}u,\quad \text{on } M
\]
having exactly $\ell$ nodal domains and having least energy among all such solutions.
\end{corollary}

\begin{proof}
Since the maximum principle is valid for the operator $P_g=-\Delta_g + \frac{N-2}{4(N-1)}R_g$, the least energy $\Gamma$-invariant fully nontrivial solution to the system \eqref{Eq:Q Systems} given by Theorem \ref{Th:MainQSystems} can be taken to be positive in each of its components \cite[Theorem 3.4 a)]{ClappSzulkin19}. In this way, each component of the functions $\overline{u}_k$, defined in the proof of Theorem \ref{Th:MainQSystems}, is nonnegative, yielding that $u_{\infty,i}$ is also nonnegative for every $i=1,\ldots,\ell$ and $\Omega_i=\{x\in M\smallsetminus(M_-\cup M_+) \;:\; u_{\infty,i}>0\}$. The rest of the proof is, up to minor details, the same as in the proof of item (iii) of Theorem 4.1 in \cite{CSS21}.
\end{proof}

\begin{remark}
As it was already noticed in Remark \ref{Remark:OptimalPartitionPowerP}, Theorem \ref{Theorem:OptimalPartition} holds true for $Q$-curvature type equations with power nonlinearities $2\leq p\leq 2^\ast_m$. Hence, the previous corollary is also true for Yamabe-type problems of the form
\[
-\Delta_g u + R u = \vert u\vert^{p-2}u,\quad \text{on } M,
\]
where $2\leq p\leq 2_1^\ast$ and $R\in \mathcal{C}^\infty(M)$ is positive and $\Gamma$-invariant.\qed
\end{remark}
\medskip

\section{Examples of cohomogeneity one actions satisfying hypothesis of Theorem \ref{Theorem:OptimalPartitionSymmetry}} \label{Section:Examples}

In this section we will see concrete examples in which Theorem \ref{Theorem:OptimalPartitionSymmetry} can be applied. First, we will discuss cohomogeneity one actions where the metric decomposition (\hyperref[Gamma:MetricDecomposition]{$\Gamma3$}) holds true.
\newline

Let $\Gamma$ be a closed subgroup of isometries of $(M,g)$ inducing a cohomogeneity one action. As before, the principal orbits of the action correspond to the hypersurfaces given by the regular level sets of $r$, and if $K$ denote the principal isotropy, all of them are diffeomorphic to $\Gamma/K$ (see \cite[Proposition 6.41]{AlexBettiol}). Hence, we can fix one of these level hypersurfaces, say $M_{d/2}:= r^{-1}(d/2)$. When we have a cohomogeneity one action by isometries, we can describe the decomposition of the metric $g$ in terms of a one-parameter family of metrics on $M_{d/2}$ as follows. It is known that, given a minimizing horizontal geodesic between the $M_+$ and $M_{-}$, a $\Gamma$-invariant metric $g$ away from the singular orbits can be written as:
\[
g=dt^2+g_t,
\]
for $t\in (0, d)$, where $g_t$ is a smooth family of homogeneous metrics on $\Gamma/K=M_{d/2}$. Let $\mathfrak{g}$,  $\mathfrak{k}$ be the Lie algebras of $\Gamma$ and $K$, respectively. Let $\mathfrak{m}$ be the orthogonal complement of $\mathfrak{k}$ in $\mathfrak{g}$, with respect to a bi-invariant metric $B$ on $\mathfrak{g}$. Since $\Gamma$ is compact, then it admits such a metric. There is a natural identification of $\mathfrak{m}$ with the tangent space to a principal orbit $M_{d/2}$. Namely, for each $X\in \mathfrak{m}$,
\begin{equation*}
 %\label{eqn: funvec}
X_p^*:=\left.\frac{d}{dt}\right|_{t=0} \exp(t X)\cdot p
\end{equation*}
is a tangent vector at $p\in M_{d/2}$. Then $g_t$ corresponds to a $1$-parameter family of invariant inner products on $\mathfrak{m}$. 
\newline

One way to obtain the metric decomposition (\hyperref[Gamma:MetricDecomposition]{$\Gamma3$}) is when $\mathfrak{m}$ decomposes into $k$ mutually orthogonal Ad$(K)$-invariant subspaces:
\begin{equation}
    \label{PrincipalOrbitDecomposition}
    \mathfrak{m}=\mathfrak{m}_1\oplus \cdots \oplus \mathfrak{m}_k
\end{equation}
such that the metric $g_t$ can be written as:
\[
    g_t=\sum_{j=1}^k f_j^2(t)\ B|_{\mathfrak{m}_j},
\]
for some positive smooth functions $f_j$, $j=1, \dots, k$ on $(0, d)$, and satisfying some smoothness conditions at $0$ and $d$. Such conditions describe a compactification of $M_{d/2}\times (0, d)$ by adding two compact submanifolds, corresponding to the endpoints of $(0, d)$.
\newline

These types of cohomogeneity one metrics are called {\em diagonal metrics}. A decomposition in this form can be obtained in several settings:
\begin{itemize}

 \item If $\Gamma$ is a simple Lie group (i.e., if $\mathfrak{g}$ is simple), by the Schur's Lemma \cite[Theorem 4.29]{Hall15}.
 \medskip
 
    \item If $\Gamma$ is a semi-simple Lie group. It follows from the Weyl's Theorem on complete reducibility, which ensures that one obtains a decomposition (\ref{PrincipalOrbitDecomposition}) with irreducible factors, with respect to the adjoint representation; and by the Schur's Lemma. See \cite[Theorem 7.8]{Hall15}.
    \medskip

    \item If the Killing form of $\mathfrak{g}$ is negative definite. More generally, if the adjoint representation of $\mathfrak{g}$ is unitary, one also has such a decomposition. See \cite[Proposition 4.27]{Hall15}. 
\end{itemize}
\medskip

Denote $d_j:= \dim\mathfrak{m}_j$, $j=1, \dots, k$, then the volume form of $g$ is given by:
\[
dV_g=\prod_{j=1}^k f_j^{d_j}\ dt\ dV_{(B|_{\mathfrak{m}_j})},
\]
and we recover the formula in Lemma \ref{Lemma:MetricVolumeDecomposition}. 
\newline

Recall also a fundamental fact: A cohomogeneity one action can be determined through a group diagram $K\subset \{K_{+}, K_{-}\}\subset \Gamma$, provided that $K_{\pm}/K$ are spheres (see Section 6.3 from \cite{AlexBettiol}). 
\newline

We develop this theory for the following well-known example of an isometric action on the round sphere that has been used in several papers to obtain sign-changing solutions to semilinear elliptic problems (see, for instance, \cite{Ding1986,BaScWe,ClappFdz17,FdzPetean20,CSS21,ClappFernandezSaldana2021}).
\newline

\ex  Consider the sphere $(\mathbb{S}^N,g)$ with its canonical metric, which is an Einstein metric with positive scalar curvature. Let $n_1,n_2\geq 2$ be integers such that $n_1+n_2=N+1$. Set
\begin{eqnarray*}
\Gamma=O(n_1)\times O(n_2), &\quad& K=O(n_1-1)\times O(n_2-1)\\ K_{+}=O(n_1-1)\times O(n_2), &\quad& K_{-}=O(n_1)\times O(n_2-1),
\end{eqnarray*}
where $O(n)$ is the group of linear isometries of $\mathbb{R}^n$. Note that we are regarding $\Gamma$ as acting on $\mathbb{S}^{n_1-1}\times \mathbb{S}^{n_2-1}$, trivially in one component, and with the transitive action by rotations in the other one. Then we obtain two possible isotropy groups $K_{\pm}$. A similar approach considers $K$ as a subgroup of $K_{\pm}$. In those cases, the isotropy is a copy of $O(n_1-1)$ or $O(n_2-1)$, for each corresponding case.  Using that the $(n-1)$-sphere can be described as the quotient $\mathbb{S}^{n-1}\simeq O(n)/O(n-1)$, we obtain the following quotients:
\[
\begin{array}{cc}
     \Gamma/K_{+}=\mathbb{S}^{n_1-1},&  \Gamma/K_{-}=\mathbb{S}^{n_2-1},\\
     \Gamma/K=\mathbb{S}^{n_1-1}\times \mathbb{S}^{n_2-1},& K_{+}/K=\mathbb{S}^{n_2-1},\quad K_{-}/K=\mathbb{S}^{n_1-1}.
\end{array}
\]
Hence, the group diagram $K\subset \{K_{+}, K_{-}\}\subset \Gamma$ defines a cohomogeneity one action of $\Gamma=O(n_1-1)\times O(n_2-1)$ on $\mathbb{S}^N$ with singular orbits $\mathbb{S}^{n_1-1}$ and $\mathbb{S}^{n_2-1}$, and with principal orbit $\mathbb{S}^{n_1-1}\times \mathbb{S}^{n_2-1}$. The orbit space is diffeomorphic to $[0, \pi]$. Therefore, this shows that conditions \hyperref[Gamma:Cohomogeneity]{$(\Gamma1)$} and \hyperref[Gamma:DimensionOrbits]{$(\Gamma2)$} are fulfilled.
\newline

The Lie algebra $\mathfrak{g}$ of $\Gamma$ is isomorphic to $\mathfrak{so}(n_1)\oplus \mathfrak{so}(n_2)$, where $\mathfrak{so}(n)$ denotes the Lie algebra of the $n\times n$ skew-symmetric matrices. It is a simple Lie algebra of dimension $n(n-1)/2$, except for the case $\mathfrak{so}(4)$ which is semi-simple. In this last case, its decomposition into simple factors is:
\[
\mathfrak{so}(4)=\mathfrak{so}(3)\oplus \mathfrak{so}(3).
\]

As previously, denote by $\mathfrak{m}$ the ${\rm Ad}(K)$-invariant complement of $\mathfrak{k}$ in $\mathfrak{g}$, where $\mathfrak{k}$ is the Lie algebra of $K$. It is canonically identified with the tangent space at $eK$, $T_{eK}\Gamma/K\simeq \mathbb{S}^{n_1-1}\times \mathbb{S}^{n_2-1}$. Then $\mathfrak{m}$ can be decomposed into two factors $\mathfrak{p}_1$ and $\mathfrak{p}_2$ given by:
\[
\mathfrak{p}_1\simeq\mathfrak{so}(n_1)/\mathfrak{so}(n_1-1), \quad \mbox{and\quad }\mathfrak{p}_2\simeq\mathfrak{so}(n_2)/\mathfrak{so}(n_2-1),
\]
with ${\rm dim}(\mathfrak{p}_1)=n_1-1$ and ${\rm dim}(\mathfrak{p}_2)=n_2-1$. If $n_1=4$ (or $n_2=4$), then
\[
\mathfrak{p}_1\simeq\left(\mathfrak{so}(3)\oplus \mathfrak{so}(3)\right)/\mathfrak{so}(3)\simeq \mathfrak{so}(3).
\]
Therefore, in any case an invariant metric $g$ on $\mathbb{S}^N$ can be written as:
\[
g=dt^2+f_1(t)^2B|_{\mathfrak{p}_1}+f_2(t)^2B|_{\mathfrak{p}_2}
\]
where $B$ is a bi-invariant metric on $\mathfrak{g}$. We may then take:
\[
f_1(t)=\cos(t/2), \quad f_2(t)=\sin(t/2).
\]
Observe that these functions satisfy the smoothness conditions (\ref{SmoothnessCond}). Then condition \hyperref[Gamma:MetricDecomposition]{$(\Gamma3)$} is also satisfied.
\medskip

The mean curvature $h(t)$ of the principal orbit is:
\begin{eqnarray*}
h(t)=(n_1-1)\frac{f_1'(t)}{f_1(t)}+(n_2-1)\frac{f_2'(t)}{f_2(t)}&=&-\frac{(n_1-1)}{2}\frac{\sin(t/2)}{\cos(t/2)}+\frac{(n_2-1)}{2}\frac{\cos(t/2)}{\sin(t/2)}\\
&=&\frac{1}{2}\frac{(n_2-1)\cos^2(t/2)-(n_1-1)\sin^2(t/2)}{\sin(t/2)\cos(t/2)}\\
&=&\frac{2(n_1+n_2-2)\cos(t)}{\sin(t)}-\frac{2(n_2-n_1)}{\sin(t)}.
\end{eqnarray*}
Then the volume of the principal orbits along $(0, \pi)$ is:
\[
2|\mathbb{S}^{n_1-1}||\mathbb{S}^{n_2-1}| \cos^{n_1-1}(t/2)\sin^{n_2-1}(t/2)
\]
where $|\mathbb{S}^{n_i-1}|$ is the $(n_i-1)$-dimensional measure of the sphere $\mathbb{S}^{n_i-1}$, for $i=1, 2$. This setting was under consideration in \cite{ClappFernandezSaldana2021}, where the authors studied  the system (\ref{Eq:Q Systems}) on $\mathbb{R}^N$ and on $\mathbb{S}^N.$\qed
\newline 

%\ex %(see Section 4 from \cite{Kraines66}); in $\mathbb{CP}^2\# \overline{\mathbb{CP}^2}$ with the Page metric \cite{Page1978}; in $\mathbb{HP}^2\# \overline{\mathbb{HP}^2}$ with the B\"{o}hm metric \cite{Bohm98}; in $\mathbb{CP}^2\# 2\overline{\mathbb{CP}^2}$ with the Chen-LeBrun-Weber metric\cite{CLW08}. 

\ex Let $(\mathbb{CP}^N, g_{FS})$ be the complex projective space with the Fubini-Study metric. First, recall that $\mathbb{CP}^N=\mathbb{S}^{2N+1}/{\rm U}(1)$, and $\mathbb{S}^{2N+1}\subset \mathbb{R}^{2N+2}$. Write $N=2+k$ for $k\in\mathbb{N}\setminus\{0\}$. Note that we may decompose $\mathbb{R}^{2N+2}\equiv\mathbb{C}^{n_1}\times \mathbb{C}^{n_2}$, where $n_1=n_2=2$ if $k=1$, and $n_1=k$, $n_2=3$ if $k\geq 2$. 

Consider the action of $\Gamma= {\rm U}(n_1)\times {\rm U}(n_2)$ on $\mathbb{C}^{n_1}\times \mathbb{C}^{n_2}$, where ${\rm U}(n_1)$ acts on $\mathbb{C}^{n_1}$ by unitary transformations and trivially on $\mathbb{C}^{n_2}$; and analogously when taking ${\rm U}(n_2)$. Thus we obtain a $\Gamma$-action on $\mathbb{CP}^N$ that can be lifted to an action on $\mathbb{S}^{2N+1}$ that commutes with the diagonal action of ${\rm U}(1)$. Since the complex projective space of dimension $n$ can be written as $\mathbb{CP}^{n}={\rm U}(n+1)/({\rm U}(n)\times {\rm U}(1))$, and the $2n+1$-sphere as $\mathbb{S}^{2n+1}={\rm U}(n+1)/{\rm U}(n)$, the groups
\begin{eqnarray*}
\Gamma={\rm U}(n_1)\times {\rm U}(n_2), &\quad& K={\rm U}(n_1-1)\times {\rm U}(n_2-1)\times {\rm U}(1)\\
K_{+}={\rm U}(n_1-1)\times {\rm U}(n_2)\times {\rm U}(1), &\quad& K_{-}={\rm U}(n_1)\times {\rm U}(n_2-1)\times {\rm U}(1)
\end{eqnarray*}
induce a cohomogeneity one action on $\mathbb{CP}^N$. The orbits are diffeomorphic to 
\[
\Gamma/K=\mathbb{CP}^{n_1-1}\times \mathbb{CP}^{n_2-1}, \quad 
\Gamma/K_{+}=\mathbb{CP}^{n_1-1}, \quad \Gamma/K_{-}=\mathbb{CP}^{n_2-1}.
\]
The Fubini-Study can be written as:
\[
g_{FS}=dt^2+g_t=dt^2+f_1g|_{\mathcal{H}} +f_2g|_{\mathcal{V}}
\]
where $g$ is the round metric on $\mathbb{S}^{2N-1}$, while $\mathcal{H}$ and $\mathcal{V}$ are the horizontal and vertical spaces of the Hopf bundle $\mathbb{S}^1\to \mathbb{S}^{2N-1}\to \mathbb{CP}^{N-1}$, 
with:
\[
f_1(t)=\sin(t), \quad f_2(t)=\sqrt{\frac{2N-2}{N}}\sin(t)\cos(t).
\]
Notice that dim$(\mathcal{H})=2N-2$ and dim$(\mathcal{V})=1$. See also \cite[Example 6.52]{AlexBettiol}.

This exhibits that condition \hyperref[Gamma:MetricDecomposition]{$(\Gamma3)$} is also satisfied.
On the other hand, by means of the Koszul formula we have, 
\[\frac{d}{dt} g_t(X, Y)=2 g_t(L_t(X), Y)\]
where $L_t$ is the shape operator on the principal orbit $\Gamma/K$, that takes a diagonal form:
\[
L_t=-\left(\begin{array}{cc} I_{2N-2} \cdot \cot(t) &0\\ 0 & 2\cot(2t)
\end{array} \right),
\]
where $I_{2N-2}$ is the identity matrix of $(2N-2)\times (2N-2)$. Thus, the mean curvature of the principal orbit $(\Gamma/K, g_t)$ is:
\[
h(t)=2 \cot (2t)+(2N-2)\cot (t).
\]
The volume form also takes a product form as in Lemma \ref{Lemma:MetricVolumeDecomposition}.
\qed
\newline

\ex The symmetric metric on $\mathbb{HP}^N$, given in \cite[Section 4]{Kraines66} can be described in a similar way as in the complex case. 

For an integer $n$, let denote by ${\rm Sp}(n)$ the compact symplectic group. Recall that $\mathbb{HP}^N=\mathbb{S}^{4N+3}/{\rm Sp}(1)$, and $\mathbb{S}^{4N+3}\subset \mathbb{R}^{4N+4}$. Write $N=4+k$ for $k\in\mathbb{N}\cup\{0\}$ and  $\mathbb{R}^{4N+4}\equiv\mathbb{H}^{n_1}\times \mathbb{H}^{n_2}$, where $n_1=n_2=3$ if $k=1$, and $n_1=k$, $n_2=5$ if $k\geq 2$. Consider the action of $\Gamma= {\rm Sp}(n_1)\times {\rm Sp}(n_2)$ on $\mathbb{H}^{n_1}\times \mathbb{H}^{n_2}$, where ${\rm Sp}(n_1)$ acts on $\mathbb{H}^{n_1}$ and trivially on $\mathbb{H}^{n_2}$; and analogously with ${\rm Sp}(n_2)$. Thus we obtain a $\Gamma$-action on $\mathbb{HP}^N$ that can be lifted to an action on $\mathbb{S}^{4N+3}$ that commutes with the action of ${\rm Sp}(1)$. Since 
 $\mathbb{HP}^{n}={\rm Sp}(n+1)/({\rm Sp}(n)\times {\rm Sp}(1))$, and $\mathbb{S}^{2n+1}={\rm Sp}(n+1)/{\rm Sp}(n)$, the groups
\begin{eqnarray*}
\Gamma= {\rm Sp}(n_1)\times {\rm Sp}(n_2), &\quad& K={\rm Sp}(n_1-1)\times {\rm Sp}(n_2-1)\times {\rm Sp}(1)\\
K_{+}={\rm Sp}(n_1-1)\times {\rm Sp}(n_2)\times {\rm Sp}(1), &\quad& K_{-}={\rm Sp}(n_1)\times {\rm Sp}(n_2-1)\times {\rm Sp}(1)
\end{eqnarray*}
induce a cohomogeneity one action on $\mathbb{HP}^N$. The orbits are diffeomorphic to 
\[
\Gamma/K=\mathbb{HP}^{n_1-1}\times \mathbb{HP}^{n_2-1}, \quad 
\Gamma/K_{+}=\mathbb{HP}^{n_1-1}, \quad \Gamma/K_{-}=\mathbb{HP}^{n_2-1}.
\]
Then the standard metric can be written as:
\[
g=dt^2+g_t=dt^2+f_1g|_{\mathcal{H}} +f_2g|_{\mathcal{V}}
\]
where $g$ is the round metric on $\mathbb{S}^{4N-1}$, while $\mathcal{H}$ and $\mathcal{V}$ are the horizontal and vertical spaces of the Hopf bundle $\mathbb{S}^3\to \mathbb{S}^{4N-1}\to \mathbb{HP}^{N-1}$. In this case, 
\[
f_1(t)=\sin(t)\cos (t), \quad f_2(t)=\cos(t),
\]
and dim$(\mathcal{H})=4N-4$ and dim$(\mathcal{V})=3$. The shape operator and the mean curvature can be computed similarly as in the previous case.
\qed
\ex The Page metric is a cohomogeneity one Einstein metric $g=dt^2+g_t$ on $\mathbb{CP}^2\#\overline{\mathbb{CP}^2}$, under the action of ${\rm U}(2)$\cite{Page1978}, with positive scalar curvature. The principal orbits are diffeomorphic to $\mathbb{S}^3$, and the 1-parameter family of invariant metrics $g_t$ on $\mathbb{S}^3$ splits through the classical Hopf fibration:
\[
g_t=f_1^2(t)g_{\mathbb{S}^1}+f_2^2(t)g_{\mathbb{S}^2}.
\]
It can be smoothly extended to the singular orbits that are diffeomorphic to $\mathbb{S}^2$. Here $g_{\mathbb{S}^1}$ and $g_{\mathbb{S}^1}$ are the canonical metrics on ${\mathbb{S}^1}$ and ${\mathbb{S}^2}$, respectively. The group diagram of the action is the same as in Example 6.8.\qed
\newline

\ex The next cases are described in \cite{Bohm98}, and follow a general scheme of construction.
Let $\Gamma$ be a compact Lie group, and let $K\subset  K_{+}=K_{-}$ be subgroups of $\Gamma$, such that $\mathbb{S}^{d_S}=K_{\pm}/K$ is a sphere of dimension $d_S\geq 1$. C. Böhm in \cite[Sections 1 and 2]{Bohm98} gives a description of the initial value problem that allows to obtain cohomogeneity one closed Einstein manifolds $(M, g)$, with positive scalar curvature, satisfying the following:
\begin{itemize}
\item The group that acts is $\Gamma$, and the two singular orbits are the same and have positive dimension. 

\item The space of $\Gamma$-invariant metrics on the principal orbits $\Gamma/K$ is two dimensional, which implies that the metric can be written as:
\[
g=dt^2+f_1(t)^2g^S+f_2^2(t)\bar{g}
\]
where $g^S$ is the canonical metric on the sphere $\mathbb{S}^{d_S}$, and $\bar{g}$ is a $\Gamma$-invariant metric on the singular orbits $\Gamma/K_{\pm}$. 
\end{itemize}

Böhm listed in \cite[Table 1]{Bohm98} known manifolds that are obtained through this path. They include $\mathbb{CP}^{N}$, $\mathbb{HP}^{N}$, flag manifolds {$F^N$} and the Cayley plane $Ca\mathbb{P}^2$.  It is worth mentioning that the underlying actions are different from those actions given in Examples 6.2 and 6.3. They are given by the standard action of unitary groups on $\mathbb{CP}^{N}$; by symplectic groups on $\mathbb{HP}^{N}$ and $F^N$; and by the $Spin(9)$ group on $Ca\mathbb{P}^2$. 
\medskip

Therefore all of these examples fulfill conditions \hyperref[Gamma:Cohomogeneity]{$(\Gamma1)$}, \hyperref[Gamma:DimensionOrbits]{$\Gamma2$} and \hyperref[Gamma:MetricDecomposition]{$(\Gamma3)$}.\qed
\newline

\ex  
 Along with the same scheme described in Example 6.5, Böhm also constructed a cohomogeneity one Einstein metric on $\mathbb{HP}^2\#\overline{\mathbb{HP}^2}$, under the action of ${\rm Sp}(2)$, and with positive scalar curvature. The principal orbits are diffeomorphic to $\mathbb{S}^7$ and the singular orbits are both diffeomorphic to $\mathbb{S}^4$. See \cite[Theorem 3.5]{Bohm98}. Here $\overline{\mathbb{HP}^2}$ denotes $\mathbb{HP}^2$ with the opposite orientation.\qed
 \newline
 
\ex The scheme described in \cite{Bohm98} is exploded to obtain more general examples. If $\Gamma/K$ is a compact, connected, isotropy irreducible homogeneous space, that is not a torus, having $1< {\rm dim}(\Gamma/K)\leq 6$ and $3\leq k+1\leq 9-{\rm dim}(\Gamma/K)$, then Böhm proved that $\mathbb{S}^{k+1}\times \Gamma/K$ admits infinitely many non-isometric cohomogeneity one Einstein metrics with positive scalar curvature \cite[Theorem 3.4]{Bohm98}. \qed
\newline

Theorem \ref{Theorem:OptimalPartitionSymmetry} also holds on manifolds with positive scalar curvanture (not necessarily Einstein manifolds), with an action satisfying  \hyperref[Gamma:Cohomogeneity]{$(\Gamma1)$}--\hyperref[Gamma:MetricDecomposition]{$(\Gamma3)$}. Among others, the Koiso-Cao soliton metric will serve to provide explicit examples of this situation.

\ex The Koiso-Cao soliton is a metric constructed on $\mathbb{CP}^2\#\overline{\mathbb{CP}^2}$. Set
\[
\Gamma=U(2), \ K=U(1), \ K_{+}=K_{-}=U(1)\times U(1)
\]
Since $K_{\pm}/K=U(1)=\mathbb{S}^1$, these groups induce a cohomogeneity one action on a manifold $M$ of dimension $4$, whose principal orbits are diffeomorphic to $\mathbb{S}^3$, and the singular orbits are both diffeomorphic to $\mathbb{S}^2$. By the decomposition \cite[Proposition 6.33]{AlexBettiol} $M$ is equivariantly diffeomorphic to the $\mathbb{S}^2$-bundle:
\[
M\simeq \mathbb{S}^3\times_{\mathbb{S}^1}\mathbb{S}^2\bigcup_{\mathbb{S}^3} \mathbb{S}^3\times_{\mathbb{S}^1}\mathbb{S}^2\simeq \mathbb{S}^3\times_{\mathbb{S}^1}\mathbb{S}^2.
\]
where $\mathbb{S}^1$ is acting on $\mathbb{S}^2=\{(z, x)\in \mathbb{C}\times \mathbb{R}|\ |z|^2+x^2=1\}$ by:
\[
\mathbb{S}^1\times \mathbb{S}^2\to \mathbb{S}^2,\quad (e^{i\theta}, (z, x))\mapsto (e^{i\theta}z , x).
\]
Thus $\mathbb{S}^3\times_{\mathbb{S}^1}\mathbb{S}^2$ is the associated bundle to the classical Hopf fibration $\mathbb{S}^1\to \mathbb{S}^3\to \mathbb{S}^2$. Furthermore, recall that the only $\mathbb{S}^2$-bundles over $\mathbb{S}^2$ are $\mathbb{S}^2\times \mathbb{S}^2$  or $\mathbb{CP}^2\#\overline{\mathbb{CP}^2}$. Therefore $M=\mathbb{CP}^2\#\overline{\mathbb{CP}^2}$. We denote the orbit space by $M/{\rm U}(2)=[\alpha, \beta]$.
\newline

Following \cite{TOR17}, consider $f_2:[\alpha, \beta]\to \mathbb{R}$ the positive smooth solution to the equation:
\begin{equation*}
%\label{eqn:SolY}
2f_2f_2''+4f_2'^2-4+f_2^2(1+cf_2'^2)=0.
\end{equation*}
satisfying $f_2'(\alpha)=f_2'(\beta)=0$, and $f_2(\alpha)f_2''(\alpha)=-f_2(\beta)f_2''(\beta)=-1$. Here $c$ is the unique root of the function 
\[
\xi(x) = e^{2x} (2- 4x+ 3x^2)- 2 + x^2.
\]
See \cite[Lemma 4.1]{Cao1}. Furthermore, observe that:
\[
\xi(-2/3)=\frac{6}{e^{4/3}}-\frac{14}{9}>0,\  \mbox{and}\ \xi(-1/2)=\frac{19}{4 e}-\frac{7}{4}<0 
\]
which implies that $-1<c<-1/2$. One also may obtain $f_2(\alpha)=\sqrt{6}$ and $f_2(\beta)=\sqrt{2}$. Set $g_{\mathbb{S}^1}$ and $g_{\mathbb{S}^2}$ the round metrics on $\mathbb{S}^1$ and $\mathbb{S}^2$, respectively. It was shown in \cite[Section 2]{TOR17} that setting $f_1=-f_2f_2'$, the metric
\[
g=dt^2+f_1^2(t)g_{\mathbb{S}^1}+f_2^2(t)g_{\mathbb{S}^2},   
\]
with the above conditions on $f_2$, determines a U$(2)$-invariant Kähler-Ricci soliton on $\mathbb{CP}^2\#\overline{\mathbb{CP}^2}$. The work of X.-J. Wang and X. Zhu \cite{Wang-Zhu04} directly implies that this must be the Koiso-Cao soliton.
\newline

The shape operator $L_t$ takes the block diagonal form:
\[
L_t=\left(\begin{array}{cc} \frac{f_1'}{f_1} &0\\ 0 & \frac{f_2'}{f_2}I_{2}\end{array} \right),
\]
where $I_{2}$ is the identity matrix of dimension $2$. Thus the mean curvature $h(t)$ is given by:
\begin{equation*}
%\label{eqn: meancur}
h(t)=\frac{f_1'}{f_2}+2\frac{f_2'}{f_2}.
\end{equation*}
It can be easily verified that $h$ is smooth away from the singular orbits.
\newline

The Hopf fibration may be used to describe the Koiso-Cao soliton as a local product $\mathbb{S}^1\times U\times (\alpha, \beta)$, with $U$ an open subset of $\mathbb{S}^2$. The volume of the principal orbit is $-f_2^3f_2'$. Indeed, using the conditions on $f_2$, the volume of $g$ is easily computed:
\begin{eqnarray*}
{\rm Vol}({g})&=&-2\pi^2\int_{\alpha}^{\beta}f_2^3f_2'\ dt\nonumber\\
&=&16\pi^2.
\end{eqnarray*}

Even more, the radial Ricci curvature is:
\begin{eqnarray*}
    Ric_g(\nabla f, \nabla f)=f'^2Ric\left(\frac{\partial }{\partial t}, \frac{\partial }{\partial t}\right)>0.
\end{eqnarray*} 
The scalar curvature of the Koiso-Cao soliton can be computed in terms of the function $f_2$. First the potential $f$ of the Ricci soliton is given by $
f=-\frac{cf_2^2}{2}$ (see \cite[equation (2.6)]{TOR17}). Recall that by averaging with the action in the Ricci soliton equation (\ref{eqn: RicciSoliton}), we may obtain that the scalar curvature $R_g$ is a U$(2)$-invariant function. Taking the trace in (\ref{eqn: RicciSoliton}), we have
\[
R_g-\Delta_g f=4
\]
and therefore the scalar curvature is:
\begin{equation*}
%\label{eqn:Scalar}
R_g=4cf_2'^2+2cf_2 f_2''+4.
\end{equation*}
By \cite[Proposition 3.1]{TOR17}, $R_g$ is a positive decreasing function in $[\alpha, \beta]$ satisfying:
\[
\max_{[\alpha, \beta]}R_g=R_g(\alpha)=4-2c>0, \quad \min_{[\alpha, \beta]}R_g= R_g(\beta)=4+2c>0.
\]\qed

\ex Besides the aforementioned examples, there is another general form to obtain a cohomogeneity one manifold satisfying conditions \hyperref[Gamma:Cohomogeneity]{$(\Gamma1)$}--\hyperref[Gamma:MetricDecomposition]{$(\Gamma3)$} and hypothesis of Theorem \ref{Theorem:OptimalPartitionSymmetry}. \medskip

A classical result by J. Milnor states that any Lie group $\Gamma$ with compact universal covering admits an Einstein metric $\hat{g}$ of positive scalar curvature \cite[Corollary 7.7]{Milnor76}. Therefore, if $(M, g)$ is a cohomogeneity one Einstein manifold (under the isometric action of $\Gamma$) of positive scalar curvature, then the Riemannian product $(M\times \Gamma, g+\hat{g})$ is also a cohomogeneity one Einstein manifold of positive scalar curvature. For this, we may take, for instance, compact simply connected Lie groups SL$(n, \mathbb{C})$, SU$(n)$, $n\geq1$; or compact Lie groups such as U$(n)$, SO$(n)$, or ${\rm Sp}(n)$. In particular, if the $\Gamma$-action on  $(M, g)$ satisfies \hyperref[Gamma:Cohomogeneity]{$(\Gamma1)$}--\hyperref[Gamma:MetricDecomposition]{$(\Gamma3)$}, then the product manifold $(M\times \Gamma, g+\hat{g})$ also fulfills these conditions.\qed
\newline

On the other hand, recall that the $Q$-curvature of a Riemannian manifold $(M, g)$ with dimension $N\geq 4$ is given by:
%\[
%Q_g:=-\frac{1}{2(N-1)}\Delta R_g+ \frac{N^3-4N^2+16N-16}{8(N-1)^2(N-2)^2}R_g^2-\frac{2}{(N-2)^2}|Ric_g|^2,
%\]
\begin{equation}
\label{eqn: Qv3}
Q_g=-{\bf a}\Delta_g R_g+ {\bf b}R_g^2-{\bf c}|Ric_g|^2,
\end{equation}
where $|Ric_g|_g^2:=\sum_{i, j=1}^N|R_{ij}|^2$, for $R_{ij}$ the components of the Ricci tensor, whit ${\bf a}, {\bf b}$ and ${\bf c}$ the positive dimensional constants:
\[
{\bf a}:=\frac{1}{2(N-1)}, \quad {\bf b}:=\frac{N^3-4N^2+16N-16}{8(N-1)^2(N-2)^2}, \quad {\bf c}:=\frac{2}{(N-2)^2}.
\]
\medskip

\begin{lemma}
Let ${\bf a}, {\bf b}$ and ${\bf c}$ the above defined dimensional constants. Seen as functions in $N$, the following are true:
\begin{enumerate}
\item $(2{\bf a}-{\bf c})(N)>0$ if and only if $N>4$, and $(2{\bf a}-{\bf c})(4)=-1/6$. 
\item $({\bf b}-2{\bf a}+{\bf c})(N)>0$ if and only if $N=4$ or $N=5$.

\end{enumerate}
\end{lemma}

Suppose that $(M, g)$ is a  Riemannian product $(M:=M_1\times M_2, g:=g_1+g_2)$ with $N=n_1+ n_2$. We have the following identities of its Laplacian and its Ricci and scalar curvatures: 
\begin{eqnarray*}
\Delta_g=\Delta_{g_1}+\Delta_{g_2}, &\quad&
Ric_g=Ric_{g_1}+ Ric_{g_2},\\
R_g=R_{g_1}+R_{g_2}, &\quad& |Ric_g|^2=|Ric_{g_1}|^2+|Ric_{g_1}|^2.
\end{eqnarray*}

\begin{prop}
Assume that $(M_1, g_1)$ is the Koiso-Cao soliton, and $(M_2, g_2)$ is any homogeneous Einstein manifold with  ${\rm dim}(M_2)=n_2\geq 4$ and positive scalar curvature. Therefore the $Q$-curvature of $(M, g)$ is positive.
\end{prop}
\begin{proof}
From hypothesis we have ${\rm dim}(M)\geq 8$. Observe that the dimensional constants ${\bf a}, {\bf b}$ and ${\bf c}$ are decreasing, as functions on $[4, \infty)$. In particular, ${\bf a}(4)>{\bf a}(N).$

Let $f$ be the Ricci potential of $(M_1, g_1)$. Since $R_{g_2}$ is constant, $\Delta_{g_1}R_{g_2}=\Delta_{g_2}R_{g_2}=0$; and $R_{g_2}^2=n_2|Ric_{g_2}|^2$. From the identity (\ref{eqn: LaplaRiccisoliton}), and Theorem \ref{thm: QcurvKoisoCao} given below, we obtain that the $Q$-curvature of $(M, g)$ satisfies:
\begin{eqnarray*}
Q_g&=&-{\bf a}(N)\Delta_gR_g+{\bf b}(N)R_g^2-{\bf c}(N) |Ric_g|^2\\
&=&-{\bf a}(N)\Delta_{g_1}R_{g_1}+{\bf b}(N)R_g^2-{\bf c}(N)|Ric_{g_1}|^2-{\bf c}(N)|Ric_{g_2}|^2\\
&=&\begin{split} 2{\bf a}(N)|Ric_{g_1}|^2-2{\bf a}(N)R_{g_1}-2{\bf a}(N)Ric_{g_1}(\nabla f, \nabla f)+{\bf b}(N) R_g^2-{\bf c}(N) |Ric_{g_1}|^2\\
-{\bf c}(N) |Ric_{g_2}|^2
\end{split}\\
&=&\begin{split}(2{\bf a} -{\bf c})(N)|Ric_{g_1}|^2-2{\bf a}(N)R_{g_1}-2{\bf a}(N)Ric_{g_1}(\nabla f, \nabla f)+{\bf b}(N) R_{g}^2\\-{\bf c}(N)|Ric_{g_2}|^2
\end{split}\\
&>&\begin{split}(2{\bf a} -{\bf c})(4)|Ric_{g_1}|^2-2{\bf a}(4)R_{g_1}-2{\bf a}(4)Ric_{g_1}(\nabla f, \nabla f)+{\bf b}(N) R_{g}^2-{\bf c}(N)|Ric_{g_2}|^2
\end{split}\\
&=&Q_{g_1}-{\bf b}(4)R_{g_1}^2+2{\bf b}(N)R_{g_1}R_{g_2}+{\bf b}(N)R_{g_1}^2+{\bf b}(N)R_{g_2}^2-\frac{{\bf c}(N)}{n_2}R_{g_2}^2.
\end{eqnarray*}
Note here that, to obtain the inequality, we are using that $(2{\bf a} -{\bf c})(4)=-1/6$, while $(2{\bf a} -{\bf c})(N)>0$ for $N>4$. Observe that for any $N\geq8$, the difference 
\begin{eqnarray*}
{\bf b}(N)-\frac{{\bf c}(N)}{n_2}={\bf b}(N)-\frac{{\bf c}(N)}{N-4}&=&\frac{N^4-8 N^3+16 N^2-48 N+48}{8 (N-4) (N-2)^2 (N-1)^2}\\
&=&\frac{ N^2(N-4)^2-48 (N-1)}{8 (N-4) (N-2)^2 (N-1)^2}
\end{eqnarray*}
is positive, since $N^2(N-4)^2>48 (N-1)$. We may rescale the metric $g_2$ so that $R_{g_2}=\frac{{\bf b}(4)}{2{\bf b}(N)}\max_{M_1}R_{g_1}$, which implies that $-{\bf b}(4)R_{g_1}^2+2{\bf b}(N)R_{g_1}R_{g_2}>0$. 
\medskip

In Theorem \ref{thm:QKoiso-Cao_positive} it is proved that $Q_{g_1}>0$. From this, we conclude that $Q_g>0$. \end{proof}
\medskip

We keep the notations of this propositoin for the next example.  
\medskip

\ex Recall that in dimension at least $6$, the coercivity of the Paneitz operator follows if both the $Q$-curvature and scalar curvature are positive, by Proposition \ref{Proposition:Coercivity}. Thus, by the previous result, the product  $(M=M_1\times M_2, g=g_1+g_2)$ of dimension $N=4+ n_2\geq 8$, has coercive Paneitz operator.
\medskip

If $\Gamma_2$ is a compact Lie group acting isometric and transitively on $(M_2, g_2)$, then U$(2)\times \Gamma_2$ acts by cohomogeneity one on $M$, and satisfies conditions \hyperref[Gamma:Cohomogeneity]{$(\Gamma1)$}--\hyperref[Gamma:MetricDecomposition]{$(\Gamma3)$}.\qed 
\newline

Apart from the previous manifolds, there are several known examples of cohomogeneity one Einstein metrics of positive scalar curvature. In the sequel of works \cite{KoSa85, KoSa88}, N. Koiso and Y. Sakane gave examples of Kähler-Einstein metrics with positive scalar curvature and arbitrary cohomogeneity. In \cite{WangZiller90}, M. Y. Wang and W. Ziller constructed Einstein manifolds with positive scalar
curvature and cohomogeneity one (or bigger), in odd-dimension. Large families of examples were given by C. Boyer, K. Galicki, and B. Mann, in \cite{BGM94} (see also the references therein). They constructed examples of compact Einstein manifolds of simply-connected compact inhomogeneous Einstein manifolds of positive scalar curvature in dimensions $4n-5$, for $n>2$.

% %%%%%%%%%%%%%%%%%%%%%%%%%%%%%%%%%%%%%%%%%%%
% %%%%%%%%%%%%%%%%%%%%%%%%%%%%%%%%%%%%%%%%%%
% %%%%%%%%%%%%%%%%%%%%%%%%%%%%%%%%%%%%%%

%%%%%%%%%%%%%%%%%%%%%%%%%%%%%%%
%%%%%%%%%%%%%%%%%%%%%%%%%%%%%%%
%%%%%%%%%%%%%%%%%%%%%%%%%%%%%%%%%

\section{{\em Q}-curvature on Ricci solitons and coercivity of the Paneitz operator}\label{Section:Q-curvatureRicciSolitons}

Recall that a closed Riemannian manifold $(M, g)$ is said to be a Ricci soliton if there exists a function $f\in C^{\infty}(M)$ and a scalar $\mu\in\{-1, 0, 1\}$ satisfying the differential equation:
\begin{equation*}
Ric_g+\hess(f)+\mu g=0.
\end{equation*}
The Ricci soliton is called {\em shrinking}, {\em steady} or {\em expanding} if $\mu=1, 0$ or $-1$, respectively.
\newline

It is our intention to extend our results to Ricci soliton metrics, for higher order GJMS operators. The next step is to study the coercivity of the Paneitz operator. Then, in terms of Proposition \ref{Proposition:Coercivity}, it is desirable to obtain conditions on a Ricci soliton that ensure the positivity of both its $Q$-curvature and its scalar curvature. Nevertheless, we might ask for conditions that in fact could not happen. We will motivate a context where the positivity holds, in terms of general properties and rigidity results of Ricci solitons. We will subsequently consider the Koiso-Cao soliton as the leading metric where we can obtain explicit computations.
\newline

For the main properties of Ricci solitons, we refer to the survey \cite{Cao2010}, and the references therein. It is well known that closed Ricci solitons of constant scalar curvature are Einstein metrics. Even more, steady or expanding closed Ricci solitons have constant scalar curvature. With respect to the scalar curvature, it is known that Ricci soliton metrics that do not reduce as Einstein metrics must have positive scalar curvature. Thus, the class of Ricci solitons under our interest are closed shrinking Ricci solitons of positive scalar curvature. Furthermore, as it was explained in the Introduction, it is required to have radially positive Ricci curvature, otherwise the Ricci soliton might be trivial. 
\newline

For any Ricci soliton $g$ with Ricci potential $f$ and constant $\mu$, it is known that: 
\begin{equation*}
\Delta_g f=R_g+n\mu, %\quad\quad |\nabla f|^2=2\mu\cdot f-R_g,
\end{equation*}
and the so called {\em conservation law}:
\begin{equation}
\label{id: conlaw}
R_g+|\nabla f|^2-2\mu\cdot f=D
\end{equation}
for some constant $D$. It is also known that the Laplacian of the scalar curvature is given by:
\begin{equation}
\label{eqn: LaplaRiccisoliton}
  \Delta_g R_g=2\mu R_g-2|Ric_g|^2+2Ric_g(\nabla f, \nabla f),
\end{equation}
and that 
\[
\nabla_{\nabla f}R_g=2Ric_g(\nabla f, \nabla f).
\]
See \cite[Lemma 2.5]{PetWyl09a}.
\newline

\begin{theorem}
\label{thm: QcurvKoisoCao}
The $Q$-curvature of a shrinking Ricci soliton $(M, g)$, with Ricci potential $f$, is given by:
\[ 
Q_g=(2{\bf a}-{\bf c})|Ric_g|_g^2+{\bf b}R_g^2-2{\bf a}R-2{\bf a}\ Ric_g(\nabla f, \nabla f).
\]
\end{theorem}
\begin{proof}
By formula (\ref{eqn: LaplaRiccisoliton}), we have that if $(M, g)$ is a shrinking soliton, with $\mu=1$, 
\[
\Delta_g R_g=2R_g-2|Ric_g|_g^2+2Ric_g(\nabla f, \nabla f).
\]
If we directly substitute this into (\ref{eqn: Qv3}), on a shrinking Ricci soliton we have:
\[
Q_g=(2{\bf a}-{\bf c})|Ric_g|_g^2+{\bf b}R^2-2{\bf a}R-2{\bf a}\ Ric_g(\nabla f, \nabla f).
\]\end{proof}

Therefore, we may obtain:
\begin{prop}
Let $(M, g)$ be a shrinking Ricci soliton of dimension $N\geq 4$, with potential $f$. Suppose that it has radially positive Ricci curvature, then it has positive $Q_g$-curvature provided:
\begin{itemize}
\item For $N=4$, if $R^2>|Ric_g|_g^2+2R_g+2 Ric_g(\nabla f, \nabla f)$.

\item Or for $N>4$, if $(2{\bf a}-{\bf c})|Ric_g|_g^2>2{\bf a}\ Ric_g(\nabla f, \nabla f)>0$ and $R_g> 2{\bf a}/{\bf b}$.
\end{itemize}
\end{prop}
\medskip

\begin{proof}[Proof of Theorem \ref{Theorem:Q-curvatureRicciSolitons}]
It is a direct consequence of the previous results.
\end{proof}
\bigskip

Now we will compute the $Q$-curvature of the Koiso-Cao soliton. Recall that this soliton can be given in terms of the smooth positive solution to the equation:
\begin{equation}
\label{eqn:SolR}
2f_2f_2''+4f_2'^2-4+f_2^2(1+cf_2'^2)=0
\end{equation} 
satisfying:
\[
f_2'(\alpha)=f_2'(\beta)=0,\quad f_2(\alpha)f_2''(\alpha)=-f_2(\beta)f_2''(\beta)=-1,\quad f_2(\alpha)=\sqrt{6},\ f_2(\beta)=\sqrt{2}.
\]
From \ref{eqn:SolR}, we find that the conservation law (\ref{id: conlaw}) is equivalent to:
\begin{equation}
\label{id: conlawKC}
f'^2=2f-R_g+4+4c,
\end{equation}
where $f=-\frac{cf_2^2}{2}$ is the Ricci potential of the soliton.
\newline

\begin{lemma}
    The solution to (\ref{eqn:SolR}), $f_2:[\alpha, \beta]\to \mathbb{R}$, satisfies $f_2'^2<1/2$. 
\end{lemma}
\begin{proof}
The maximum value of $f_2'^2$ on $[\alpha, \beta]$ is achieved if and only if $f_2'=0$ or $f_2''=0$; both cases cannot happen simultaneously. Consider those points where $f_2''=0$. Thus, if $\tau\in[\alpha, \beta]$ is a critical point of $f_2'^2$, then:
\begin{eqnarray*}
R_g(\beta)=4+2c>R_g(\tau)&=&4cf_2'^2(\tau)+2cf_2(\tau)
f_2''(\tau)+4\\
&=&4cf_2'^2(\tau)+4.
\end{eqnarray*}
Therefore $4cf_2'^2>2c$, and since $c<-1/2$, the bound follows.\end{proof}

%its critical value is given by
%\[
%f_2'^2(\tau)=\frac{4-f_2^2(\tau)}{4+cf_2^2(\tau)},
%\]
%which is strictly positive. Set $s:=f_2'^2(\tau)$. Therefore, 
%\[
%f_2^2(\tau)=\frac{4-4s}{1+s c}.
%\]
%Since $2\leq f_2^2$, then $2<f_2^2(\tau)$. From this we obtain:
%\[
%2(1+sc)<4-4s
%\]
%which is equivalent to
%\[
%2s+sc<1
%\]
\bigskip

%On any closed Riemannian manifold $(M, g)$ of dimension $N\geq 3$, let ${\bf E}=\{E_1, \dots, E_N\}$ be an orthonormal frame in which the Ricci tensor have a diagonal form. Then we have:
%\[
%|Ric_g|_g^2=\sum_{i=1}^NR_{ii}^2, \quad R_g=\sum_{i=1}^NR_{ii}.
%\]

\begin{theorem}
\label{thm:QKoiso-Cao_positive}
The Koiso-Cao soliton has positive $Q$-curvature.
\end{theorem}
\begin{proof}
Following equations (2.7)--(2.9) from \cite[Section 2]{TOR17}, the non-zero components of the Ricci tensor are:
\begin{eqnarray*}
    R_{11}&=&R_{22}=Ric\left(\frac{\partial }{\partial t}, \frac{\partial }{\partial t}\right)=1+c(f_2f_2''+f_2'^2),\\ R_{33}&=&R_{44}=Ric\left(\frac{Y}{f_2}, \frac{Y}{f_2}\right)=1+cf_2'^2.
\end{eqnarray*}
Here $Y$ denotes the SU$(2)-$left invariant vector field given by $
 Y(v,  w)=( w, -v),$ where $(v, w)$ are coordinates in $\mathbb{C}^2$. The radial Ricci curvature is given by:
\begin{eqnarray*}
    Ric_g(\nabla f, \nabla f)=f'^2Ric\left(\frac{\partial }{\partial t}, \frac{\partial }{\partial t}\right)=f'^2R_{11}.
 \end{eqnarray*}
 Since $R_g=2R_{11}+2R_{33}$, and $|Ric_g|_g^2=2R_{11}^2+2R_{33}^2$, the $Q$-curvature of the Koiso-Cao soliton is given by:
\begin{eqnarray*}
    6Q_g&=&R_g^2-|Ric_g|^2-2R_g-2Ric_g(\nabla f, \nabla f)\\    &=&2|Ric_g|^2+8R_{11}R_{33}-|Ric_g|_g^2-2R_g-2f'^2R_{11}\\
    &=&|Ric_g|_g^2+8R_{11}R_{33}-2R_g-2f'^2R_{11}.
\end{eqnarray*}
Furthermore, using the conservaton law (\ref{id: conlawKC}), we have:
\begin{eqnarray*}
    6Q_g &=&2 R_{11}(1+c (f_2f_2''+f_2'^2))+2R_{33}(1+c f_2'^2)+8R_{11}R_{33}-2R_g-2f'^2R_{11}\\
    &=&R_g(1+c f_2'^2)+R_{11}(2cf_2f_2''+8R_{33}-4-2f'^2)-4R_{33}\\    &=&R_gR_{33}-4R_{33}+R_{11}(2cf_2f_2''+8R_{33}-4-2f'^2)\\
&=&R_{33}(R_g-4)+R_{11}(R_g+4cf_2'^2-2f'^2)\\
&=&R_{33}(R_g-4)+R_{11}(R_g+4cf_2'^2+4cf_2'^2f)
\end{eqnarray*}
On the other hand, note that:
\[
R_{33}(R_g-4)\geq 2cR_{33}\geq2c,
\]    
Therefore we have:
\begin{eqnarray*}
    6Q_g    &\geq& 2c+R_{11}(R_g+4cf_2'^2+4cf_2'^2f)\\
    &\geq&2c+R_{11}(4+2c+4cf_2'^2-12c^2f_2'^2)>0\end{eqnarray*}
Furthermore, by evaluating at $\alpha$ and $\beta$, we obtain:
\[
6Q_g(\alpha)=2 (1-c)^2+4 (1-c)-2\approx 8.77772, \quad 6Q_g(\beta)=2 (1+c)^2+4 (1+c)-2\approx 0.335809.
\]
\end{proof}
\begin{figure}[h]
\caption{$Q$-curvature of the Koiso-Cao soliton}
\centering
\includegraphics[scale=0.85]{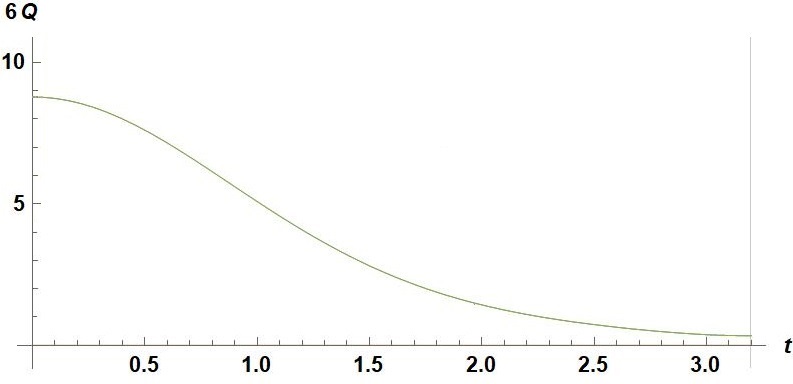}
\end{figure}

\section*{Acknowledgments}
 The authors want to thank Mónica Clapp for suggesting us to provide explicit examples of cohomogeneity Einstein manifolds having positive scalar curvature.  We also thank Matthew Gursky for pointed us out reference \cite{UhlenbeckViaclovsky2000} and for his useful comments about interior regularity for higher-order conformal operators. Finally, we want to thank Paul Yang for useful conversations that contribute to the proof of the symmetric unique continuation principle given in Lemma \ref{Lemma:UniqueContinuation}.

\bibliographystyle{plainurl}

%\bibliography{bibnods}

\end{document}